\documentclass[11pt,leqno]{article}
\usepackage{a4wide}
\usepackage{amssymb,upgreek}
\usepackage{amsmath}
\usepackage{amsthm}
\usepackage{mathrsfs}
\usepackage{euscript,csquotes}

\usepackage[curve]{xypic}
\usepackage[usenames,dvipsnames]{xcolor}
\usepackage{endnotes,comment}

\usepackage{hyperref}

\theoremstyle{plain}

\newcommand{\Tr}{\operatorname{Tr}}
\newcommand{\Fr}{\operatorname{Fr}}
\newcommand{\Sh}{\operatorname{Sh}}

\newcommand{\ev}{\operatorname{ev}}
\newcommand{\rec}{\operatorname{rec}}
\newcommand{\gr}{\operatorname{gr}}
\newcommand{\pr}{\operatorname{pr}}

\newcommand{\Hom}{\operatorname{Hom}}
\newcommand{\End}{\operatorname{End}}
\newcommand{\Ext}{\operatorname{Ext}}

\newcommand{\Mod}{\operatorname{Mod}}
\newcommand{\Ind}{\operatorname{Ind}}
\newcommand{\ind}{\operatorname{ind}}
\newcommand{\inj}{\operatorname{inj}}
\newcommand{\pro}{\operatorname{pro}}

\newcommand{\cores}{\operatorname{cores}}

\newcommand{\Tor}{\operatorname{Tor}}

\newcommand{\chara}{\operatorname{char}}

\def\anti{\EuScript J}

\newtheorem{theorem}{Theorem}[section]
\newtheorem{corollary}[theorem]{Corollary}
\newtheorem{lemma}[theorem]{Lemma}
\newtheorem{proposition}[theorem]{Proposition}

\newtheorem{definition}[theorem]{Definition}

\newtheorem{example}[theorem]{Example}

\theoremstyle{definition}
\newtheorem{remark}[theorem]{Remark}
\title{\textbf{dg-Hecke duality and tensor products}}
\author{Peter Schneider, Claus Sorensen}
\date{}

\begin{document}
\maketitle

\begin{abstract}
We continue our study of the monoidal category $D(G)$ begun in \cite{SS}. At the level of cohomology we transfer the duality functor $R\underline{\Hom}(-,k)$ to the derived category of dg-modules $D(H_U^{\bullet})$. In the process we develop a more general and streamlined approach to the anti-involution $\anti$ from \cite{OS}. We also verify that the
tensor product on $D(G)$ corresponds to an operadic tensor product on the dg-side (cf \cite{KM}). This uses a result of Schn\"{u}rer on dg-categories with a model structure.
\end{abstract}



\section{Introduction and background}

\subsection{Notation}

The following notation will remain in force throughout this article. We let $G$ be a fixed $p$-adic Lie group of dimension $d$ over $\Bbb{Q}_p$, and we let $k$ be a field of characteristic $p$. We denote by $\Mod(G)$ the abelian category of smooth $G$-representations on $k$-vector spaces and by $D(G)$ its (unbounded) derived category.

\subsection{The monoidal structure on $D(G)$ and dg Hecke modules}

In our first paper \cite{SS} Cor.\ 3.2 we have shown that $D(G)$ is a closed symmetric monoidal category. The product is the naive tensor product $V_1^\bullet \otimes_k V_2^\bullet$ of complexes with the diagonal $G$-action; the unit object is the trivial representation $k$. The internal Hom-functor $R\underline{\Hom}$ is the total derived functor
\begin{equation*}
  R\underline{\Hom}(V_1^\bullet, V_2^\bullet) = \varinjlim_K \Hom_{\Mod(K)}^\bullet (V_1^\bullet, J^\bullet)
\end{equation*}
where $V_2^\bullet \xrightarrow{\sim} J^\bullet$ is a homotopically injective resolution and where the inductive limit runs over all compact open subgroups $K \subseteq G$. Of particular interest is the (derived) duality functor $R\underline{\Hom}(-,k)$.

Throughout the paper we fix an open subgroup $U \subseteq G$ which is pro-$p$ and torsion free. In \cite{SS} we introduced the full subcategory $D(G)^a$ of globally admissible complexes  $V^\bullet$ which are those for which each cohomology group $H^i(U,V^\bullet)$, for $i \in \mathbb{Z}$, is a finite dimensional $k$-vector space. By \cite{SS} Cor.\ 4.3 and Prop.\ 4.5, $D(G)^a$ is the largest subcategory of $D(G)$ on which the duality functor $R\underline{\Hom}(-,k)$ restricts to an involution (and independent of $U$).

We now consider the compact induction $\mathbf{X}_U := \ind_U^G(k)$ in $\Mod(G)$ together with the differential graded $k$-algebra $H_U^\bullet$ which is defined as the opposite of the differential graded endomorphism algebra of a fixed injective resolution $\mathbf{X}_U \xrightarrow{\simeq} \mathcal{I}^\bullet$ in $\Mod(G)$ of the representation $\mathbf{X}_U$ (cf.\ \cite{DGA} \S 3). It has nonzero cohomology at most in degrees $0$ up to $d$, and $h^0(H_U^\bullet) = H_U = k[U \backslash G / U]$ is the usual Hecke algebra of the pair $(G,U)$. We let $D(H_U^\bullet)$ denote the derived category of left differential graded $H_U^\bullet$-modules. The main result (Thm.\ 9) in \cite{DGA} says that the functor $RH^0(U,-)$ lifts in a natural way to an equivalence of triangulated categories
\begin{equation*}
  H : D(G) \xrightarrow{\;\simeq\;} D(H_U^\bullet)
\end{equation*}
making the diagram
\begin{equation*}
  \xymatrix{
  D(G) \ar[dr]_{RH^0(U,-)} \ar[r]^{H}
                & D(H_U^\bullet) \ar[d]^{\mathrm{cohomology}}  \\
                & D(k)             }
\end{equation*}
commutative. We introduce the strictly full triangulated subcategory $D_{fin}(H_U^\bullet)$ of $D(H_U^\bullet)$ consisting of all objects all of whose cohomology vector spaces are finite dimensional. By definition the equivalence $H$ restricts to an equivalence
\begin{equation*}
  H : D(G)^a \xrightarrow{\;\simeq\;} D_{fin}(H_U^\bullet) \ .
\end{equation*}

We may transfer the closed symmetric monoidal structure of $D(G)$ via the equivalence $H$ to $D(H_U^\bullet)$. For this we recall the precise definition of $H$ and its quasi-inverse $T$. Let $K(G)$ denote the category of unbounded complexes in $\Mod(G)$ with homotopy classes of of chain maps as morphisms and $K_{\inj}(G)$ the full triangulated subcategory of homotopically injective complexes. The obvious functor $q_G : K(G) \rightarrow D(G)$ restricts to an equivalence $q_G : K_{\inj}(G) \xrightarrow{\simeq} D(G)$, and we let $\mathbf{i}$ denote a quasi-inverse of the latter. Correspondingly, we have the homotopy category $K(H_U^\bullet)$ of differential graded left $H_U^\bullet$ and its subcategory $K_{\pro}(H_U^\bullet)$ of homotopically projective modules; the obvious functor $q_H : K(H_U^\bullet) \rightarrow D(H_U^\bullet)$ restricts to an equivalence $q_H : K_{\pro}(H_U^\bullet) \xrightarrow{\simeq} D(H_U^\bullet)$, and we denote by $\mathbf{p}$ a quasi-inverse functor. Then
\begin{equation*}
  H(V^\bullet) = q_H(\Hom_{\Mod(G)}^\bullet(\mathcal{I}^\bullet, \mathbf{i}V^\bullet))  \quad\text{and}\quad   T(M^\bullet) = q_G(\mathcal{I}^\bullet \otimes_{H_U^\bullet} \mathbf{p}M^\bullet) \ .
\end{equation*}
It will be convenient in the following to also fix a (homotopy) injective resolution $k \xrightarrow{\simeq} \mathcal{J}^\bullet$ of the trivial $G$-representation $k$.

\begin{example}
We immediately see that $\mathbf{U}:=H(k) = q_H(\Hom_{\Mod(G)}^\bullet(\mathcal{I}^\bullet, \mathcal{J}^\bullet))$ is the unit object in $D(H_U^\bullet)$; its cohomology is the group cohomology $\Ext_{\Mod(G)}^*(\mathbf{X}_U,k) = H^*(U,k)$.
\end{example}

Part of the motivation behind this paper was to describe the resulting monoidal structure on $D(H_U^\bullet)$ in intrinsic terms. We come close to achieving this goal as we explain next.

\subsection{A summary of our main results}

Let $V^{\bullet}$ be an object of $D(G)$ as above. The corresponding dg module $H(V^{\bullet})$ has cohomology $H^*(U,V^{\bullet})$ which is a graded right module over the Yoneda algebra $E^*=\Ext_{\Mod(G)}^*(\mathbf{X}_U,\mathbf{X}_U)$. Ideally one would like a description of $H(R\underline{\Hom}(V^{\bullet},k))$ as a certain intrinsic dg dual of
$H(V^{\bullet})$. This is given by our functor $\Delta$ in section \ref{sec:duality-functor}, but it is admittedly not a very workable definition. Instead, after passing to cohomology, we obtain a satisfying description of the $E^*$-module $H^*(U,R\underline{\Hom}(V^{\bullet},k))$ as a certain dual of $H^*(U,V^{\bullet})$. This is our Theorem \ref{maindual} which we restate here with a different notation:

\medskip

\noindent {\bf{Theorem A}}. {\it{There is an isomorphism of graded right $E^*$-modules}}
$$
H^*(U,R\underline{\Hom}(V^{\bullet},k))\simeq \Hom_k^{\text{gr}}(H^*(U,V^{\bullet}),k)[-d]^{\anti \otimes \chi_G}
$$

\medskip

We refer the reader to section \ref{subsec:Delta} for more details on the right-hand side of this isomorphism. Here we will just highlight that the superscript $\anti \otimes \chi_G$ means
we turn the right-hand side into a {\it{right}} $E^*$-module via an anti-involution which is a twist of (a generalization of) the one in \cite{OS}. The arguments leading up to \cite[Prop.~2.7]{SS} essentially prove Theorem A for the Hecke algebra $E^0$. The result for the full Yoneda algebra $E^*$ is a vast generalization thereof, and the proof is significantly more involved.

The anti-involution $\anti: E^* \longrightarrow E^*$ is at the heart of the argument. It was previously studied in \cite{OS} for a $p$-adic reductive group with $U$ being a pro-$p$ Iwahori subgroup. In Section \ref{subsec:coho-anti} we give a more conceptual definition of $\anti$ which applies to any pair of groups $(G,U)$ as above.

A crucial intermediate step of the proof of Theorem A is the special case where $V^{\bullet}$ is the representation ${\bf{X}}_U$ concentrated in degree zero (Corollary \ref{cor:swap-Jchi}). A bit of unwinding shows this amounts to $\anti\otimes \chi_G$ being dual to the ${\bf{X}}_U$-factor switch on $\Ext_{\Mod(G)}^*({\bf{X}}_U \otimes_k {\bf{X}}_U,k)$.

To explain our results on the transfer of the tensor product, let $V^\bullet$ and $W^\bullet$ be two objects from $D(G)$. We are aiming for a description of $H(V^\bullet\otimes_k W^\bullet)$ as a tensor product of the dg modules $H(V^\bullet)$ and $H(W^\bullet)$. This exploits a certain dg $(H_U^\bullet,H_U^\bullet \otimes_k H_U^\bullet)$-bimodule $H_U^\bullet(2)$ (cf. Definition \ref{dgbi}). The result is reminiscent of the operadic tensor product discussed in \cite{KM}.
Here is a reformulation of Theorem our \ref{tensor}:

\medskip

\noindent {\bf{Theorem B}}. {\it{There is a natural isomorphism in $D(H_U^\bullet)$,}}
$$
H_U^\bullet(2) \otimes_{H_U^\bullet \otimes_k H_U^\bullet}^L \big(H(V^\bullet) \otimes_k H(W^\bullet) \big) \simeq H(V^\bullet\otimes_k W^\bullet).
$$

\medskip

At the level of cohomology this yields an Eilenberg-Moore spectral sequence (Corollary \ref{EMSS}) computing $H^*(U,V^\bullet\otimes_k W^\bullet)$ in terms of
$H^*(U,V^\bullet)$ and $H^*(U,W^\bullet)$.

The very definition of $H_U^\bullet(2)$ is a bit subtle. When $G$ is compact $\Hom_{\Mod(G)}^{\bullet}(\mathcal{I}^\bullet, \mathcal{I}^\bullet \otimes_k \mathcal{I}^\bullet)$ works.
The problem is $\mathcal{I}^\bullet \otimes_k \mathcal{I}^\bullet$ need not be an injective resolution of
${\bf{X}}_U \otimes_k {\bf{X}}_U$ when $G$ is non-compact, in which case we replace $\mathcal{I}^\bullet \otimes_k \mathcal{I}^\bullet$ by an injective resolution. However, to retain the
$H_U^\bullet \otimes_k H_U^\bullet$-module structure we need to resolve $\mathcal{I}^\bullet \otimes_k \mathcal{I}^\bullet$ in a dg functorial way. This can be done due to a general result of Schn\"{u}rer \cite{Schn} on dg categories with a model structure.\\

\noindent\textbf{Acknowledements:} The first author acknowledges support by the Deutsche Forschungsgemeinschaft
(DFG, German Research Foundation) – Project-ID 427320536 – SFB 1442, as well as
under Germany’s Excellence Strategy EXC 2044 390685587,
Mathematics Münster: Dynamics–Geometry–Structure.


\section{Transfer of the duality functor}\label{sec:duality-functor}

\subsection{The differential graded situation}

We first discuss the special case of the duality functor $R\underline{\Hom}(-,k)$. Inserting the definitions we have
\begin{align*}
  \Delta(M^\bullet) & := H(R\underline{\Hom}(T(M^\bullet),k)) \\
   & = q_H(\Hom_{\Mod(G)}^\bullet(\mathcal{I}^\bullet, \mathbf{i}\underline{\Hom}^\bullet(\mathcal{I}^\bullet \otimes_{H_U^\bullet} \mathbf{p}M^\bullet,\mathcal{J}^\bullet)))
\end{align*}
In the proof of \cite{SS} Prop.\ 4.5 we observed that with $\mathcal{J}^\bullet$ also the complex $\underline{\Hom}^\bullet(\mathcal{I}^\bullet \otimes_{H_U^\bullet} \mathbf{p}M^\bullet,\mathcal{J}^\bullet)$ is homotopically injective. We deduce that
\begin{align*}
  \Delta(M^\bullet) & = q_H(\Hom_{\Mod(G)}^\bullet(\mathcal{I}^\bullet, \underline{\Hom}^\bullet(\mathcal{I}^\bullet \otimes_{H_U^\bullet} \mathbf{p}M^\bullet,\mathcal{J}^\bullet)))  \\
    & = q_H(\Hom_{\Mod(G)}^\bullet(\mathcal{I}^\bullet \otimes_k \mathcal{I}^\bullet \otimes_{H_U^\bullet} \mathbf{p}M^\bullet,\mathcal{J}^\bullet))  \qquad\qquad\text{by adjunction}  \\
    & = q_H(\Hom_{H_U^\bullet}(pM^\bullet,\Hom_{\Mod(G)}^\bullet(\mathcal{I}^\bullet \otimes_k \mathcal{I}^\bullet, \mathcal{J}^\bullet)))  \qquad\text{by $\Hom$-$\otimes$ adjunction}.
\end{align*}
The last term above is a differential graded left $H_U^\bullet$-module induced by the right $H_U^\bullet$-module structure of the left tensor factor $\mathcal{I}^\bullet$. In fact $\Hom_{\Mod(G)}^\bullet(\mathcal{I}^\bullet \otimes_k \mathcal{I}^\bullet, \mathcal{J}^\bullet)$ is a differential graded left $H_U^\bullet \otimes_k H_U^\bullet$-module, which by a slight abuse of notation we denote by $R\Hom_{\Mod(G)}(\mathbf{X}_U \otimes_k \mathbf{X}_U,k)$. Using \cite{Yek} Prop.\ 14.3.15 we finally obtain that the functor $\Delta$
is given by
\begin{align*}
  \Delta : D(H_U^\bullet)^{op} & \longrightarrow D(H_U^\bullet) \\
                 M^\bullet & \longmapsto R\Hom_{H_U^\bullet}(M^\bullet, R\Hom_{\Mod(G)}(\mathbf{X}_U \otimes_k \mathbf{X}_U,k)) \ .
\end{align*}
Here $R\Hom_{H_U^\bullet}$ is formed w.r.t.\ the $H_U^\bullet$-module structure coming from the right factor $\mathbf{X}_U$. The $H_U^\bullet$-module structure on the target comes from the left factor $\mathbf{X}_U$.

\begin{proposition}\label{prop:onesided-bimod}
  With respect to the $H_U^\bullet$-action through the right hand factor $\mathcal{I}^\bullet$ we have in $D(H_U^\bullet)$ a natural isomorphism
\begin{equation*}
  R\Hom_{\Mod(G)}(\mathbf{X}_U \otimes_k \mathbf{X}_U,k) \simeq \Hom_k(H_U^\bullet[d],k) \ .
\end{equation*}
\end{proposition}
\begin{proof}
The resolution quasi-isomorphisms $k \xrightarrow{\simeq} \mathcal{J}^\bullet$ and $\mathbf{X}_U \xrightarrow{\simeq} \mathcal{I}^\bullet$ induce quasi-isomorphisms
\begin{equation*}
  \Hom_{\Mod(U)}^\bullet(\mathcal{J}^\bullet,\mathcal{J}^\bullet) \xrightarrow{\simeq} \Hom_{\Mod(U)}^\bullet(k,\mathcal{J}^\bullet)
\end{equation*}
and
\begin{equation*}
  \Hom_{\Mod(G)}^\bullet(\mathcal{I}^\bullet \otimes_k \mathcal{I}^\bullet,\mathcal{J}^\bullet) \xrightarrow{\simeq} \Hom_{\Mod(G)}^\bullet(\mathbf{X}_U \otimes_k \mathcal{I}^\bullet,\mathcal{J}^\bullet) \ ,
\end{equation*}
respectively. Furthermore we have actual isomorphisms of complexes $\mathbf{X}_U \otimes_k \mathcal{I}^\bullet \cong \ind_U^G(\mathcal{I}^\bullet)$ and hence
\begin{equation*}
  \Hom_{\Mod(G)}^\bullet(\mathbf{X}_U \otimes_k \mathcal{I}^\bullet,\mathcal{J}^\bullet) \cong \Hom_{\Mod(G)}^\bullet(\ind_U^G(\mathcal{I}^\bullet),\mathcal{J}^\bullet) \cong \Hom_{\Mod(U)}^\bullet(\mathcal{I}^\bullet,\mathcal{J}^\bullet) \ ,
\end{equation*}
the last one by Frobenius reciprocity. Similarly
\begin{equation*}
  (H_U^\bullet)^{op} = \Hom_{\Mod(G)}^\bullet(\mathcal{I}^\bullet,\mathcal{I}^\bullet) \xrightarrow{\simeq} \Hom_{\Mod(G)}^\bullet(\mathbf{X}_U,\mathcal{I}^\bullet)
                         \cong \Hom_{\Mod(U)}^\bullet(k,\mathcal{I}^\bullet) \ .
\end{equation*}
Taking all these together we may define the upper pairing in the diagram
\begin{equation*}
  \xymatrix{
    \Hom_{\Mod(G)}^\bullet(\mathcal{I}^\bullet \otimes_k \mathcal{I}^\bullet,\mathcal{J}^\bullet) \ar[d]_{\simeq} \ar@{}[r]|{\times} & (H_U^\bullet)^{op} = \Hom_{\Mod(G)}(\mathcal{I}^\bullet,\mathcal{I}^\bullet) \ar[d]_{\simeq} \ar[r] & \Hom_{\Mod(U)}^\bullet(k,\mathcal{J}^\bullet) \ar@{=}[d] \\
    \Hom_{\Mod(U)}^\bullet(\mathcal{I}^\bullet,\mathcal{J}^\bullet)  \ar@{}[r]|{\times} & \Hom_{\Mod(U)}^\bullet(k,\mathcal{I}^\bullet) \ar[r] & \Hom_{\Mod(U)}^\bullet(k,\mathcal{J}^\bullet)   }
\end{equation*}
through the lower Yoneda (or composition) pairing. The Yoneda pairing obviously is $(H_U^\bullet)^{op}$-invariant w.r.t.\ the $H_U^\bullet$-action on $\mathcal{I}^\bullet$. Hence the upper pairing is $(H_U^\bullet)^{op}$-invariant w.r.t.\ the right actions on the second factor $\mathcal{I}^\bullet$ and on $H_U^\bullet$, respectively. Recall that $U$ has cohomological dimension $d$ with $H^d(U,k) \cong k$. Denoting by $\tau^{\leq d}$ the truncation functor in degree $d$ we therefore have the quasi-isomorphism of $k$-vector spaces $\tau^{\leq d}(H^0(U,\mathcal{J}^\bullet)) \xrightarrow{\simeq} H^0(U,\mathcal{J}^\bullet) = \Hom_{\Mod(U)}^\bullet(k,\mathcal{J}^\bullet)$. We fix a left inverse quasi-isomorphism $\Hom_{\Mod(U)}^\bullet(k,\mathcal{J}^\bullet) \rightarrow \tau^{\leq d}(H^0(U,\mathcal{J}^\bullet))$ and define the composite map
\begin{equation*}
  \Hom_{\Mod(U)}^\bullet(k,\mathcal{J}^\bullet) \rightarrow \tau^{\leq d}(H^0(U,\mathcal{J}^\bullet)) \xrightarrow{\pr} H^d(U,k)[-d]) \cong k[-d] \ .
\end{equation*}
Composing the above pairings with this map we obtain $(H_U^\bullet)^{op}$-invariant pairings
\begin{equation*}
  \xymatrix{
    \Hom_{\Mod(G)}^\bullet(\mathcal{I}^\bullet \otimes_k \mathcal{I}^\bullet,\mathcal{J}^\bullet) \ar[d]_{\simeq} & \times & (H_U^\bullet)^{op} = \Hom_{\Mod(G)}(\mathcal{I}^\bullet,\mathcal{I}^\bullet) \ar[d]_{\simeq} \ar[r] & k[-d] \ar@{=}[d] \\
    \Hom_{\Mod(U)}^\bullet(\mathcal{I}^\bullet,\mathcal{J}^\bullet)  & \times & \Hom_{\Mod(U)}^\bullet(k,\mathcal{I}^\bullet) \ar[r] & k[-d].   }
\end{equation*}
Applying the usual Hom-tensor adjunction we arrive at a commutative diagram of right $(H_U^\bullet)^{op}$-equivariant complex homomorphisms
\begin{equation*}
  \xymatrix{
    \Hom_{\Mod(G)}^\bullet(\mathcal{I}^\bullet \otimes_k \mathcal{I}^\bullet,\mathcal{J}^\bullet) \ar[d]_{\simeq} \ar[r] &  \Hom_k^\bullet((H_U^\bullet)^{op},k[-d]) \ar@{=}[d] \\
    \Hom_{\Mod(U)}^\bullet(\mathcal{I}^\bullet,\mathcal{J}^\bullet) \ar[r] & \Hom_k^\bullet(\Hom_{\Mod(U)}^\bullet(k,\mathcal{I}^\bullet),k[-d]).   }
\end{equation*}
But the lower horizontal arrow induces on cohomology the duality isomorphism (4) in \cite{SS}. We conclude that both horizontal arrows are quasi-isomorphisms. The top horizontal arrow can, of course, also be read as a left $H_U^\bullet$-equivariant quasi-isomorphism $\Hom_{\Mod(G)}^\bullet(\mathcal{I}^\bullet \otimes_k \mathcal{I}^\bullet,\mathcal{J}^\bullet)  \xrightarrow{\simeq} \Hom_k^\bullet(H_U^\bullet,k[-d])$.
\end{proof}

We point out that the left, resp. right, hand side of the above map in Prop.\ \ref{prop:onesided-bimod} carries an additional left, resp.\ right, $H_I^\bullet$-action. But since the map is only a quasi-isomorphism it is not clear that any of these structures can be transported to the other side. This problem disappears after passing to cohomology. The induced isomorphism on cohomology
\begin{equation*}
  \Ext^*_{\Mod(G)}(\mathbf{X}_U \otimes_k \mathbf{X}_U, k) \xrightarrow{\cong} \Hom_k(\Ext^{d-*}_{\Mod(G)}(\mathbf{X}_U,\mathbf{X}_U),k)
\end{equation*}
comes from the Yoneda pairing
\begin{align}\label{f:Y-pairing}
  \Ext^*_{\Mod(G)}(\mathbf{X}_U \otimes_k \mathbf{X}_U, k) \times \Ext^{d-*}_{\Mod(G)}(\mathbf{X}_U,\mathbf{X}_U) & \longrightarrow \Ext^d_{\Mod(G)}(\mathbf{X}_U,k) = H^d(U,k) \cong k \\
    (f,e)  \quad\ \ \qquad\qquad\qquad & \longmapsto \langle f,e \rangle := f \circ (1 \otimes e) \ .    \nonumber
\end{align}
We also have the Yoneda algebra $E^* := E^*_U := \Ext^*_{\Mod(G)}(\mathbf{X}_U,\mathbf{X}_U) = h^*((H_U^\bullet)^{op})$. Observe that:
\begin{itemize}
  \item[--] $\Ext^*_{\Mod(G)}(\mathbf{X}_U \otimes_k \mathbf{X}_U, k)$ is a right $E^* \otimes_k E^*$-module;
  \item[--] $\Ext^{d-*}_{\Mod(G)}(\mathbf{X}_U,\mathbf{X}_U) = E^{d-*}$ is an $(E^*,E^*)$-bimodule;
  \end{itemize}
It is straightforward to see that the pairing has the property that
\begin{equation}\label{f:equiv1}
  \langle f \cdot (1 \otimes \tau),e \rangle = \langle f,\tau \cdot e \rangle    \qquad\text{for any $\tau \in E^*$},
\end{equation}
which reflects the equivariance property in Prop.\ \ref{prop:onesided-bimod}. In the subsequent subsection we will introduce an anti-involution $\anti \otimes \chi_G$ of the algebra $E^*$ and show in Proposition \ref{linear} that we have
\begin{equation}\label{f:equiv2}
  \langle f \cdot (\tau \otimes 1),e \rangle = (-1)^{\deg(e)\deg(\tau)}\cdot \langle f, e \cdot (\anti \otimes \chi_G)(\tau) \rangle    \qquad\text{for any homogeneous $\tau \in E^*$} .
\end{equation}
For $\tau\in E^0$ this is part of \cite[Prop.~2.7]{SS}.


\subsection{The cohomological anti-involution}\label{subsec:coho-anti}

Until after Lemma \ref{lem:Tr-cores} the subgroup $U \subseteq G$ may be any open subgroup. Let $V_1, V_2$ be two representations in $\Mod(G)$. As usual $\Ind_U^G(-)$ denotes the full smooth induction functor (cf.\ \cite{Vig} \S I.5). We start from the following linear map
\begin{align*}
 \anti : \Hom_{\Mod(U)} (V_1, \Ind_U^G(V_2)) & \longrightarrow \Hom_{\Mod(U)} (V_1, \Ind_U^G(V_2)) \\
                            \alpha & \longmapsto \anti(\alpha)(x)(g) := g^{-1}(\alpha(g^{-1}x)(g^{-1})) \ .
\end{align*}
In order to see that it is well defined fix an $\alpha$ and an $x \in V_1$. Obviously $\anti(\alpha)(x)$ is a map $G \rightarrow V_2$. For $u \in U$ we now verify:
\begin{itemize}
  \item[--] $\anti(\alpha)(x)$ is an induced map since
\begin{align*}
  \anti(\alpha)(x)(gu) & = (gu)^{-1}(\alpha((gu)^{-1}x)((gu)^{-1})) = u^{-1} \big( g^{-1}(\alpha(u^{-1}(g^{-1}x))(u^{-1}g^{-1})) \big) \\
   & = u^{-1} \big( g^{-1}(\alpha(g^{-1}x)(u u^{-1}g^{-1})) \big) = u^{-1} \big( \anti(\alpha)(x)(g) \big) \ .
\end{align*}
  \item[--] $\anti(\alpha)$ is $U$-equivariant since
\begin{align*}
  \anti(\alpha)(ux)(g) & = g^{-1}(\alpha(g^{-1}ux)(g^{-1}))  = g^{-1}(\alpha((u^{-1}g)^{-1}x)(g^{-1})) \\
   & = (u^{-1} g)^{-1} u^{-1}(\alpha((u^{-1}g)^{-1}x)(g^{-1})) = (u^{-1} g)^{-1} (\alpha((u^{-1}g)^{-1}x)(g^{-1} u)) \\
   & = (u^{-1} g)^{-1} (\alpha((u^{-1}g)^{-1}x)((u^{-1} g)^{-1})) = \anti(\alpha)(x)(u^{-1}g) \ .
\end{align*}
In particular $\anti(\alpha)(x)$ is fixed by any open subgroup of $U$ which fixes $x$.
\end{itemize}

\begin{proposition}\label{prop:anti-inv}
  $\anti$ is an involutive linear automorphism of $\Hom_{\Mod(U)} (V_1, \Ind_U^G(V_2))$, which is functorial in $V_1$ and $V_2$.
\end{proposition}
\begin{proof}
The functoriality is clear. For $\beta := \anti(\alpha)$ and $\gamma := \anti(\beta)$ we compute
\begin{equation*}
  \gamma(x)(g) = g^{-1} (\beta(g^{-1}x)(g^{-1})) = g^{-1} (g \alpha(g g^{-1} x)(g)) = \alpha(x)(g) \ .
\end{equation*}
\end{proof}

\begin{remark}\label{rem:anti-inv-support}
 For any $h \in G$ let $\ind_U^{UhU}(V_2) \subseteq \Ind_U^G(V_2))$ denote the $U$-invariant subspace of maps supported on $UhU$. The involution $\anti$ maps $\Hom_{\Mod(U)} (V_1, \ind_U^{UhU}(V_2))$ bijectively onto $\Hom_{\Mod(U)} (V_1, \ind_U^{Uh^{-1}U}(V_2))$.
\end{remark}

We now derive from $\anti$ two further involutions.

\subsubsection{The anti-involution on the Yoneda algebra}

\textbf{A)} Here we suppose that $V_1$ is finitely generated in $\Mod(U)$. As a consequence the functor $\Hom_{\Mod(U)}(V_1,-)$ commutes with arbitrary direct sums. Therefore $\anti$ restricts to an involutive automorphism of
\begin{equation*}
  \Hom_{\Mod(U)} (V_1, \ind_U^G(V_2)) = \Hom_{\Mod(G)} (\ind_U^G(V_1), \ind_U^G(V_2)) \ ,
\end{equation*}
again denoted by $\anti$, where the equality comes from Frobenius reciprocity. The most important example is the trivial $G$-representation $V_1 = k$, where $\anti$ restricts to $ \Hom_{\Mod(G)} (\mathbf{X}_U, \ind_U^G(V_2))$.

Recall our injective resolution $\rho : k \xrightarrow{\sim} \mathcal{J}^\bullet$ in $\Mod(G)$ of the trivial representation. As explained in \cite{OS} \S 4.2 it induces to a resolution $\ind(\rho) : \mathbf{X}_U \xrightarrow{\sim} \ind_U^G(\mathcal{J}^\bullet)$ in $\Mod(G)$ by objects which are acyclic for $\Hom_{\Mod(G)}(\mathbf{X}_U,-)$. We also fix an injective resolution $\kappa : \mathbf{X}_U \xrightarrow{\sim} \mathcal{I}^\bullet$ in $\Mod(G)$. We then have a unique (up to homotopy) homomorphism of complexes $\sigma$ in $\Mod(G)$ which makes the diagram
\begin{equation*}
  \xymatrix@R=0.5cm{
                &         \ind_U^G(\mathcal{J}^\bullet) \ar[dd]^{\sigma}     \\
  \mathbf{X}_U \ar[ur]_-{\sim}^-{\ind(\rho)} \ar[dr]^-{\sim}_-{\kappa}                 \\
                &         \mathcal{I}^\bullet                 }
\end{equation*}
commutative. This leads to the diagram of $\Hom$-complexes:
\begin{equation*}
  \xymatrix{
    \Hom_{\Mod(G)}^\bullet(\mathbf{X}_U, \ind_U^G(\mathcal{J}^\bullet))  \ar[r]^-{\sigma_*}   &  \Hom_{\Mod(G)}^\bullet(\mathbf{X}_U, \mathcal{I}^\bullet)   \\
       & \Hom_{\Mod(G)}^\bullet(\mathcal{I}^\bullet, \mathcal{I}^\bullet) = (H_U^\bullet)^{op}  \ar[u]_{\kappa^*}  }
\end{equation*}
By the above observation $\anti$ induces an involutive linear automorphism of complexes $\anti^\bullet$ on $ \Hom_{\Mod(G)}^\bullet(\mathbf{X}_U, \ind_U^G(\mathcal{J}^\bullet))$.

\begin{lemma}\label{lem:quasi-iso}
  Both maps $\kappa^*$ and $\sigma_*$ are quasi-isomorphisms.
\end{lemma}
\begin{proof}
For $\kappa^*$ use \cite{Har} Lemma I.6.2 and the subsequent p.\ 65. The map $\sigma_*$ is a quasi-isomorphism because $\ind_U^G(\mathcal{J}^\bullet)$ has $ \Hom_{\Mod(G)}^\bullet(\mathbf{X}_U, -)$-acyclic terms; this is explained in detail in \cite{DGA} p.\ 450-451.
\end{proof}

\begin{remark}
Note that the functor $\Ind_U^G$ respects injective objects. Therefore, if $G$ is compact, then  we can take $\mathcal{I}^\bullet := \Ind_U^G(\mathcal{J}^\bullet) = \ind_U^G(\mathcal{J}^\bullet)$ and $\anti$ induces an involutive automorphism of the complex $\Hom_{\Mod(G)}^\bullet(\mathcal{I}^\bullet, \mathcal{I}^\bullet) = (H_U^\bullet)^{op}$. At the dg level, its behavior with respect to the Yoneda product is unclear to us.
\end{remark}

Via the quasi-isomorphisms in Lemma \ref{lem:quasi-iso} the linear automorphism $\anti^\bullet$ induces an involutive graded linear automorphism $\anti^*$  (or simply $\anti$) of the Yoneda algebra
\begin{equation*}
  E^* = \Ext^*_{\Mod(G)}(\mathbf{X}_U,\mathbf{X}_U) = h^*((H_U^\bullet)^{op}) \ .
\end{equation*}

\begin{proposition}\label{prop:coho-anti}
   $\anti^*$ is an anti-involution of the algebra $E_U$.
\end{proposition}
\begin{proof}
This is \cite{OS} Prop.\ 6.1. Although that proof is written for a pro-$p$ Iwahori subgroup $U$ it works verbatim in our generality. To see that our $\anti$ coincides with the one defined in
\cite{OS} we introduce the Shapiro map
\begin{align*}
 \Sh_h:  \Hom_{\Mod(U)} (V_1, \ind_U^{UhU}(V_2))   & \longrightarrow  \Hom_{\Mod(U_h)} (V_1,V_2^h) \\
                \alpha & \longmapsto \Sh_h(\alpha)(x):=\alpha(x)(h) \ .
\end{align*}
Here $U_h=U \cap hUh^{-1}$ and $V_2^h$ denotes the representation $hUh^{-1}\overset{\sim}{\longrightarrow} U \longrightarrow \text{Aut}_k(V_2)$.
We can express $\anti$ in terms of Shapiro maps as follows. There is a commutative diagram of isomorphisms
\begin{equation*}
  \xymatrix{
     \Hom_{\Mod(U)} (V_1, \ind_U^{UhU}(V_2)) \ar[d]_{\Sh_h}^{\simeq} \ar[r]^{\anti} &  \Hom_{\Mod(U)} (V_1, \ind_U^{Uh^{-1}U}(V_2)) \ar[d]_{\simeq}^{\Sh_{h^{-1}}} \\
    \Hom_{\Mod(U_h)} (V_1,V_2^h) \ar[r]^{h_*} & \Hom_{\Mod(U_{h^{-1}})} (V_1,V_2^{h^{-1}}).   }
\end{equation*}
Here $(h_*\beta)(x)=h\beta(hx)$. Why? We start by pointing out why $\Sh_h$ is an isomorphism. The Shapiro map is induced by
$$
 \ind_U^{UhU}(V_2) \overset{\sim}{\longrightarrow} \ind_{U_h}^U(V_2^h) \longrightarrow V_2^h.
$$
The first map takes $f$ to the function $\phi_f(u)=f(uh)$, the second $\phi \mapsto \phi(e)$ is evaluation at the identity. Now use Frobenius reciprocity. Checking the diagram commutes is an easy computation. Indeed, letting $\beta=\Sh_h(\alpha)$ we have on the one hand
$$
(h_*\beta)(x)=h\beta(hx)=h(\alpha(hx)(h)).
$$
On the other hand, by definition of $\anti$ we have
$$
\Sh_{h^{-1}}(\anti(\alpha))(x)=\anti(\alpha)(x)(h^{-1})=h(\alpha(hx)(h))
$$
as desired. The previous commutative diagram defines the anti-involution in \cite{OS}.
\end{proof}


\subsubsection{The transpose of $\anti$}

\textbf{B)} To obtain the second involution we use, for general $V_1$ and $V_2$, the two Frobenius isomorphisms (cf.\ \cite{Vig} I.5.7)
\begin{equation*}
  \Hom_{\Mod(U)}(V_2, \Ind_U^G(V_1)) \cong \Hom_{\Mod(G)}(\ind_U^G(V_2), \Ind_U^G(V_1)) \cong \Hom_{\Mod(U)}(\ind_U^G(V_2), V_1) \ .
\end{equation*}
In order to give the explicit formula for the composite we first need to introduce, for any $h \in G$ and any $v \in V_2$, the unique function $\chara_{h,U}^v \in \ind_U^G(V_2)$ which is supported on $hU$ and has the value $v$ in $h$. The group $G$ acts on these functions as follows:
\begin{equation}\label{f:g-char}
  {^g \chara}_{h,U}^v(y) = \chara_{h,U}^v(g^{-1} y) =
  \begin{cases}
  v & \text{if $y = gh$}, \\
  0 & \text{if $y \not\in ghU$}
  \end{cases}  = \chara_{gh,U}^v (y)  \qquad\text{for any $g, y \in G$}.
\end{equation}
It is straightforward to check that the composite isomorphism
\begin{equation*}
  \rec : \Hom_{\Mod(U)}(V_2, \Ind_U^G(V_1))  \xrightarrow{\ \cong\ } \Hom_{\Mod(U)}(\ind_U^G(V_2), V_1)
\end{equation*}
is given by the formula
\begin{equation*}
  \rec(\alpha)(\phi) = \sum_{g \in G/U} \alpha(\phi(g))(g^{-1})   \qquad\text{for $\phi \in \ind_U^G(V_2)$}.
\end{equation*}
The inverse satisfies
\begin{equation*}
  \rec^{-1}(\beta)(v)(g) = \beta(\chara_{g^{-1},U}^v)    \qquad\text{for $v \in V_2$ and $g \in G$}.
\end{equation*}

We now define our second functorial involution $\anti'$ through the commutativity of the diagram
\begin{equation*}
  \xymatrix{
    \Hom_{\Mod(U)}(V_2, \Ind_U^G(V_1)) \ar[d]_{\anti}  &  \Hom_{\Mod(U)}(\ind_U^G(V_2), V_1) \ar[d]^{\anti'} \ar[l]_-{\rec^{-1}} \\
    \Hom_{\Mod(U)}(V_2, \Ind_U^G(V_1)) \ar[r]^{\rec} &  \Hom_{\Mod(U)}(\ind_U^G(V_2), V_1) .  }
\end{equation*}
Using the explicit formulas for $\anti$, $\rec$, and $\rec^{-1}$ one easily computes the explicit formula
\begin{align*}
 \anti' : \Hom_{\Mod(U)} (\ind_U^G(V_2), V_1) & \longrightarrow \Hom_{\Mod(U)} (\ind_U^G(V_2),V_1) \\
                            \lambda & \longmapsto \anti'(\lambda)(\phi) := \sum_{g \in G/U} g(\lambda(\chara_{g^{-1},U}^{g\phi(g)})) \ .
\end{align*}

Next we specialize to the case $V_2 = k$ (and $V_1 = V$): Then the definition of $\anti'$ simplifies to
\begin{equation*}
  \anti'(\lambda)(\chara_{gU}) = g (\lambda(\chara_{g^{-1} U})) \ .
\end{equation*}
Here $\chara_Y$, for any open subset $Y \subseteq G$, denotes the characteristic function of $Y$.

Let $\varsigma$ denote the endomorphism of $\mathbf{X}_U \otimes_k \mathbf{X}_U$ which swaps the factors.

\begin{lemma}\label{lem:anti'-swap}
   The diagram
\begin{equation*}
  \xymatrix{
    \Hom_{\Mod(U)}(\mathbf{X}_U,V) \ar[d]_{\cong} \ar[r]^{\anti'} & \Hom_{\Mod(U)}(\mathbf{X}_U,V) \ar[d]^{\cong} \\
    \Hom_{\Mod(G)}(\ind_U^G(\mathbf{X}_U),V) \ar[d]_{\cong}  & \Hom_{\Mod(G)}(\ind_U^G(\mathbf{X}_U),V) \ar[d]^{\cong} \\
    \Hom_{\Mod(G)}(\mathbf{X}_U \otimes_k \mathbf{X}_U,V) \ar[r]^{\varsigma^*} & \Hom_{\Mod(G)}(\mathbf{X}_U \otimes_k \mathbf{X}_U,V)   }
\end{equation*}
   is commutative.
\end{lemma}
\begin{proof}
We first recall that the inverses of the isomorphisms in the columns are induced by the maps
\begin{equation*}
  \begin{aligned}
    \mathbf{X}_U & \longrightarrow \ind_U^G(\mathbf{X}_U) \\
    \chara_{gU} & \longmapsto \chara_{1,U}^{\chara_{gU}}
  \end{aligned}
  \qquad\text{and}\qquad
  \begin{aligned}
    \ind_U^G(\mathbf{X}_U) & \longrightarrow \mathbf{X}_U \otimes_k \mathbf{X}_U \\
     \phi & \longmapsto \sum_{h \in G/U} \chara_{hU} \otimes\, h\phi(h) \ .
  \end{aligned}
\end{equation*}
The composed map simply is
\begin{align*}
  \mathbf{X}_U & \longrightarrow \mathbf{X}_U \otimes_k \mathbf{X}_U \\
    \chara_{gU} & \longmapsto \chara_{U} \otimes \chara_{gU} \ .
\end{align*}
Let now $\Lambda \in \Hom_{\Mod(G)}(\mathbf{X}_U \otimes_k \mathbf{X}_U,V)$ be an element in the lower left corner of the diagram. Going through the upper left corner it is mapped to
\begin{equation*}
  \Lambda \longmapsto [\chara_{gU} \mapsto \Lambda(\chara_U \otimes \chara_{gU})] \longmapsto [\chara_{gU} \mapsto  g(\Lambda(\chara_U \otimes \chara_{g^{-1}U}))] \ .
\end{equation*}
On the other hand going through the lower right corner $\Lambda$ is mapped to
\begin{equation*}
  \Lambda \longmapsto [\chara_{hU} \otimes \chara_{gU} \mapsto \Lambda(\chara_{gU} \otimes \chara_{hU})] \longmapsto [\chara_{gU} \mapsto  \Lambda(\chara_{gU} \otimes \chara_U)] \ .
\end{equation*}
But since $\Lambda$ is $G$-equivariant we have $g(\Lambda(\chara_U \otimes \chara_{g^{-1}U})) = \Lambda(\chara_{gU} \otimes \chara_U)$.
\end{proof}


We now let $V_1$, $V_2$, and $V_3$ be three representations in $\Mod(G)$, and we consider the composition pairing
\begin{equation*}
  \Hom_{\Mod(U)}(\ind_U^G(V_2), V_3) \   \times \   \Hom_{\Mod(U)}(V_1,\ind_U^G(V_2))  \xrightarrow{\; \circ \;} \Hom_{\Mod(U)}(V_1,V_3) \ .
\end{equation*}
In order to understand the behavior of the two involutions $\anti$ and $\anti'$ with respect to this pairing we first have to refine the pairing. For this we start with the pairing
\begin{align*}
  \ind_U^G(V_1) \times \Hom_{\Mod(U)}(V_1, \ind_U^G(V_2)) & \longrightarrow \ind_U^G(V_2)   \\
                                            (\phi,D)      & \longmapsto \langle \phi,D \rangle(g) := D(g \phi(g))(g) \ .
\end{align*}
It is well defined since:
\begin{itemize}
  \item[--] If $\phi(g) = 0$ then $\langle \phi,D \rangle(g) = 0$.
  \item[--] $\langle \phi,D \rangle(gu) = D(g u \phi(gu))(gu) = D(g \phi(g))(gu) = u^{-1}(D(g \phi(g))(g)) = u^{-1} (\langle \phi, D \rangle(g))$ for any $u \in U$.
\end{itemize}
Moreover, for any $u \in U$, we have
\begin{align*}
  \langle{^u \phi},D \rangle(g) & = D(g ({^u \phi})(g))(g) = D(g \phi(u^{-1}g))(g) = D(u u^{-1} g \phi(u^{-1}g))(u u^{-1} g)  \\
    & = D(u^{-1} g \phi(u^{-1}g))(u^{-1} g) = \langle \phi,D \rangle (u^{-1} g) = ({^u \langle} \phi,D \rangle)(g) \ .
\end{align*}
This means that, for any fixed $D \in \Hom_{\Mod(U)}(V_1, \ind_U^G(V_2))$, the map
\begin{align*}
  \ind_U^G(V_1) & \longrightarrow \ind_U^G(V_2) \\
           \phi & \longmapsto \langle \phi,D \rangle
\end{align*}
is $U$-equivariant. As a consequence we obtain the refined pairing
\begin{align}\label{f:refinedpairing}
  \Hom_{\Mod(U)}(\ind_U^G(V_2),V_3) \times  \Hom_{\Mod(U)}(V_1, \ind_U^G(V_2)) & \longrightarrow \Hom_{\Mod(U)}(\ind_U^G(V_1),V_3) \\
                       (C, D) & \longmapsto  \langle C, D \rangle (\phi) := C(\langle \phi,D \rangle) \ .   \nonumber
\end{align}

Next we define a trace map. Here the idea is to start from the map
\begin{align*}
  V_1 & \longrightarrow  \ind_U^G(V_1) \\
    v & \longmapsto [g \mapsto g^{-1} v] = \sum_{g \in G/U} \chara_{g,U}^{g^{-1} v} \ ,
\end{align*}
but which obviously is not well defined unless $G$ is compact. In a formal sense it is $U$-equivariant, though. To circumvent this problem we introduce we proceed as follows. For any $U$-bi-invariant subset $X \subseteq G$ we have the $U$-invariant subspaces $\ind_U^X(V_1) \subseteq \ind_U^G(V_1)$ of those functions which are supported on $X$. Clearly $\ind_U^G(V_1) = \ind_U^X(V_1) \oplus \ind_U^{G \setminus X}(V_1)$. We now introduce the vector subspace
\begin{align*}
  \Hom_{\Mod(U)}^{fin}(\ind_U^G(V_1),V_3) := &\ \{ F \in \Hom_{\Mod(U)}(\ind_U^G(V_1),V_3) : F | \ind_U^{G \setminus X}(V_1) = 0 \\
                        & \qquad\text{for some compact $U$-bi-invariant subset $X \subseteq G$}\}.
\end{align*}
Our trace map now is defined to be the map
\begin{align*}
  \Tr : \Hom_{\Mod(U)}^{fin}(\ind_U^G(V_1),V_3) & \longrightarrow \Hom_{\Mod(U)}(V_1,V_3)   \\
                F & \longmapsto \Tr(F)(v) := \sum_{g \in G/U} F(\chara_{g,U}^{g^{-1} v}) \ .
\end{align*}
By the condition imposed on $F$ the above defining sums are finite.

\begin{lemma}\label{lem:fin}
   If $V_1$ is finitely generated as a $U$-representation then the image of the pairing \eqref{f:refinedpairing} lies in $\Hom_{\Mod(U)}^{fin}(\ind_U^G(V_1),V_3)$.
\end{lemma}
\begin{proof}
Let $C \in \Hom_{\Mod(U)}(\ind_U^G(V_2),V_3)$, $D \in \Hom_{\Mod(U)}(V_1, \ind_U^G(V_2))$, and $\phi \in \ind_U^X(V_1)$ for some compact $U$-bi-invariant subset $X \subseteq G$. We noted already the obvious that then $\langle \phi, D \rangle \in \ind_U^X(V_2)$. On the other hand our assumption on $V_1$ implies that there is another compact $U$-bi-invariant subset $Y \subseteq G$ such that $D \in \Hom_{\Mod(U)}(V_1, \ind_U^Y(V_2))$. It follows that $\langle \phi, D \rangle \in \ind_U^Y(V_2)$. We conclude that, if $X \subseteq G \setminus Y$, then $\langle \phi, D \rangle = 0$ and a fortiori $\langle C, D \rangle (\phi) = 0$.
\end{proof}

\begin{proposition}\label{prop:refined-trace}
   If $V_1$ is finitely generated as a $U$-representation then the diagram
\begin{equation*}
  \xymatrix@R=0.5cm{
                &    \Hom_{\Mod(U)}(V_1,V_3)      \\
   \Hom_{\Mod(U)}(\ind_U^G(V_2),V_3) \times  \Hom_{\Mod(U)}(V_1, \ind_U^G(V_2)) \ar[ur]^{\circ} \ar[dr]_{\langle\ ,\ \rangle}                 \\
                &    \Hom_{\Mod(U)}^{fin}(\ind_U^G(V_1),V_3)    \ar[uu]_{\Tr}             }
\end{equation*}
   is commutative.
\end{proposition}
\begin{proof}
For $C \in \Hom_{\Mod(U)}(\ind_U^G(V_2),V_3)$, $D \in \Hom_{\Mod(U)}(V_1, \ind_U^G(V_2))$ and $v \in V_1$ we compute
\begin{align*}
  \Tr(\langle C , D \rangle)(v)  & = \sum_{g \in G/U} \langle C , D \rangle (\chara_{g,U}^{g^{-1} v}) = \sum_{g \in G/U} C (\langle \chara_{g,U}^{g^{-1} v}, D \rangle) \\
      & = \sum_{g \in G/U} C \big( \sum_{h \in G/U} \chara_{h,U}^{D(h \chara_{g,U}^{g^{-1} v}(h))(h)}  \big)  \\
      & = \sum_{g \in G/U} C \big( \chara_{g,U}^{D(g \chara_{g,U}^{g^{-1} v}(g))(g)}  \big) \\
      & = \sum_{g \in G/U} C \big( \chara_{g,U}^{D(v)(g)} \big) =  C \big( \sum_{g \in G/U} \chara_{g,U}^{D(v)(g)} \big)  \\
      & = C(D(v)) \ .
\end{align*}
\end{proof}

\begin{proposition}\label{prop:J'-J}
   If $V_1$ is finitely generated as a $U$-representation then the diagram of pairings
\begin{equation*}
  \xymatrix{
    \Hom_{\Mod(U)}(\ind_U^G(V_2),V_3) \ar@{}[r]|{\times}  & \Hom_{\Mod(U)}(V_1, \ind_U^G(V_2)) \ar[d]_{\anti} \ar[r]^{\langle\ ,\  \rangle} & \Hom_{\Mod(U)}(\ind_U^G(V_1),V_3)  \\
    \Hom_{\Mod(U)}(\ind_U^G(V_2),V_3) \ar[u]_{\anti'} \ar@{}[r]|{\times} & \Hom_{\Mod(U)}(V_1, \ind_U^G(V_2)) \ar[r]^{\langle\ ,\  \rangle} & \Hom_{\Mod(U)}(\ind_U^G(V_1),V_3) \ar[u]^{\anti'}  }
\end{equation*}
   is commutative.
\end{proposition}
\begin{proof}
Let $F \in \Hom_{\Mod(U)}(\ind_U^G(V_2),V_3)$ and $G \in \Hom_{\Mod(U)}(V_1, \ind_U^G(V_2))$. We have to show that
\begin{equation*}
  \langle \anti'(F), G \rangle = \anti' \langle F, \anti(G) \rangle)
\end{equation*}
holds true in $\Hom_{\Mod(U)}(\ind_U^G(V_1),V_3)$. By evaluation in a $\phi \in \ind_U^G(V_1)$ this amounts to showing that
\begin{equation*}
  \anti'(F) \langle \phi, G \rangle ) = \anti'(\langle F, \anti(G) \rangle )(\phi)
\end{equation*}
holds true in $V_3$. By inserting the definitions the left hand side becomes
\begin{equation*}
  \sum_{g \in G/U} g(F(\chara_{g^{-1},U}^{gG(g\phi(g))(g)})) \ .
\end{equation*}
Correspondingly the right hand side becomes
\begin{equation*}
  \sum_{g \in G/U} g(\langle F,\anti(G)\rangle (\chara_{g^{-1},U}^{g\phi(g)})) = \sum_{g \in G/U} g(F(\langle \chara_{g^{-1},U}^{g\phi(g)}, \anti(G) \rangle )) \ .
\end{equation*}
For a fixed $g \in G$ we compute the element $\langle \chara_{g^{-1},U}^{g\phi(g)}, \anti(G) \rangle \in \ind_U^G(V_2)$ as follows. First of all, by definition we have
\begin{equation*}
 \langle \chara_{g^{-1},U}^{g\phi(g)}, \anti(G) \rangle (h) = \anti(G)(h \chara_{g^{-1},U}^{g\phi(g)}(h))(h) \ .
\end{equation*}
But
\begin{equation*}
  h \chara_{g^{-1},U}^{g\phi(g)}(h) =
  \begin{cases}
  0  & \text{if $h \not\in g^{-1}U$}, \\
  \phi(g) & \text{if $h = g^{-1} u$ with $u \in U$.}
  \end{cases}
\end{equation*}
Inserting this into the right hand side of the previous equation we obtain
\begin{equation*}
  \langle \chara_{g^{-1},U}^{g\phi(g)}, \anti(G) \rangle (h) =
  \begin{cases}
  0  & \text{if $h \not\in g^{-1}U$}, \\
  \anti(G)(\phi(g))(h) & \text{if $h \in g^{-1}U$}
  \end{cases}
  = \chara_{g^{-1},U}^{\anti(G)(\phi(g))(g^{-1})}(h) \ .
\end{equation*}
Hence our above right hand side is equal to
\begin{equation*}
  \sum_{g \in G/U} g(F(\chara_{g^{-1},U}^{\anti(G)(\phi(g))(g^{-1})})) \ .
\end{equation*}
This reduces our assertion to the claim that
\begin{equation*}
  \anti(G)(\phi(g))(g^{-1}) = gG(g\phi(g))(g) \ ,
\end{equation*}
which holds by definition of $\anti$.
\end{proof}

\begin{remark}\label{rem:J'-fin}
   The involution $\anti'$ respects the subspace $\Hom_{\Mod(U)}^{fin}(\ind_U^G(V_1),V_3)$.
\end{remark}
\begin{proof}
Let $F \in \Hom_{\Mod(U)}^{fin}(\ind_U^G(V_1),V_3)$ and assume that $F | \ind_U^{UhU}(V_1) = 0$ for some $h \in G$. For $\phi \in \ind_U^{Uh^{-1}U}(V_1)$ we compute
\begin{equation}\label{f:J'-fin}
  \anti'(F)(\phi) = \sum_{g \in G/U} g(F(\chara_{g^{-1},U}^{g\phi(g)})) = \sum_{g \in Uh^{-1}U/U} g(F(\chara_{g^{-1},U}^{g\phi(g)})) = 0 \ .
\end{equation}
Hence $\anti'(F) | \ind_U^{Uh^{-1}U}(V_1) = 0$.
\end{proof}

The above discussion generalizes in a straightforward way to complexes in $\Mod(G)$ and hence to $\Ext$-groups. For the latter the only points to notice are the following:
\begin{itemize}
  \item[--] If $V_2 \xrightarrow{\sim} \mathcal{I}^\bullet$ is an injective resolution in $\Mod(G)$ and if $V_1$ is finitely generated as a $U$-representation then
\begin{equation*}
  \Ext^i_{\Mod(U)}(V_1, \ind_U^G(V_2)) = h^i(\Hom_{\Mod(U)}(V_1, \ind_U^G(\mathcal{I}^\bullet))) \ .
\end{equation*}
(The reason for this is that the functor $\Ext^i_{\Mod(U)}(V_1,-)$ commutes with arbitrary direct sums in $\Mod(U)$. Namely, since $V_1$ is finite dimensional $\Hom_k(V_1,-)$ with the diagonal $U$-action is an endo-functor of $\Mod(U)$ which preserves injective objects and commutes with arbitrary direct sums. Hence $\Ext^i_{\Mod(U)}(V_1,-) = H^i(U,\Hom_k(V_1,-))$. Recall that $H^i(U,-)$ commutes with arbitrary direct sums.)
 \item[--] By using an injective resolution $V_3 \xrightarrow{\sim} \mathcal{K}^\bullet$ we may define
\begin{equation*}
  \Ext^{i,fin}_{\Mod(U)}(\ind_U^G(V_1),V_3) := h^i(\Hom_{\Mod(U)}^{fin}(\ind_U^G(V_1),\mathcal{K}^\bullet[i]) \ .
\end{equation*}
\end{itemize}

Propositions \ref{prop:refined-trace} and \ref{prop:J'-J} therefore combine into the following general result.

\begin{proposition}\label{prop:J'-J-Ext}
   If $V_1$ is finitely generated as a $U$-representation then, for any $i, j \geq 0$, the diagram of pairings
\begin{equation*}
  \xymatrix{
    \Ext^i_{\Mod(U)}(\ind_U^G(V_2),V_3)  \ar@{=}[d] \ar@{}[r]|{\times}  & \Ext^j_{\Mod(U)}(V_1, \ind_U^G(V_2)) \ar@{=}[d] \ar[r]^-{\circ} & \Ext^{i+j}_{\Mod(U)}(V_1,V_3)  \\
    \Ext^i_{\Mod(U)}(\ind_U^G(V_2),V_3) \ar@{}[r]|{\times}  & \Ext^j_{\Mod(U)}(V_1, \ind_U^G(V_2)) \ar[d]_{\anti} \ar[r]^-{\langle\ ,\  \rangle} & \Ext^{i+j,fin}_{\Mod(U)}(\ind_U^G(V_1),V_3)
    \ar[u]^{\Tr} \\
    \Ext^i_{\Mod(U)}(\ind_U^G(V_2),V_3) \ar[u]_{\anti'} \ar@{}[r]|{\times} & \Ext^j_{\Mod(U)}(V_1, \ind_U^G(V_2)) \ar[r]^-{\langle\ ,\  \rangle} & \Ext^{i+j,fin}_{\Mod(U)}(\ind_U^G(V_1),V_3)
    \ar[u]^{\anti'} \ar[d]_{\Tr} \\
    \Ext^i_{\Mod(U)}(\ind_U^G(V_2),V_3)  \ar@{=}[u] \ar@{}[r]|{\times}  & \Ext^j_{\Mod(U)}(V_1, \ind_U^G(V_2)) \ar@{=}[u] \ar[r]^-{\circ} & \Ext^{i+j}_{\Mod(U)}(V_1,V_3)
     }
\end{equation*}
   is commutative, where $\circ$ denotes the Yoneda composition pairing.
\end{proposition}

In the following the most interesting special case is $V_1 := V_2 := k$ and $V_3 := V$.

\begin{corollary}\label{cor:J'-J-Ext}
   For any $i, j \geq 0$ the diagram of pairings
\begin{equation*}
  \xymatrix{
    \Ext^i_{\Mod(U)}(\mathbf{X}_U,V)  \ar@{=}[d] \ar@{}[r]|{\times}  & \Ext^j_{\Mod(U)}(k, \mathbf{X}_U) \ar@{=}[d] \ar[r]^-{\circ} & \Ext^{i+j}_{\Mod(U)}(k,V)  \\
    \Ext^i_{\Mod(U)}(\mathbf{X}_U,V) \ar@{}[r]|{\times}  & \Ext^j_{\Mod(U)}(k, \mathbf{X}_U) \ar[d]_{\anti} \ar[r]^-{\langle\ ,\  \rangle} & \Ext^{i+j,fin}_{\Mod(U)}(\mathbf{X}_U,V)
    \ar[u]^{\Tr} \\
    \Ext^i_{\Mod(U)}(\mathbf{X}_U,V) \ar[u]_{\anti'} \ar@{}[r]|{\times} & \Ext^j_{\Mod(U)}(k, \mathbf{X}_U) \ar[r]^-{\langle\ ,\  \rangle} & \Ext^{i+j,fin}_{\Mod(U)}(\mathbf{X}_U,V)
    \ar[u]^{\anti'} \ar[d]_{\Tr} \\
    \Ext^i_{\Mod(U)}(\mathbf{X}_U,V)  \ar@{=}[u] \ar@{}[r]|{\times}  & \Ext^j_{\Mod(U)}(k, \mathbf{X}_U) \ar@{=}[u] \ar[r]^-{\circ} & \Ext^{i+j}_{\Mod(U)}(k,V)
     }
\end{equation*}
   is commutative.
\end{corollary}

Next we analyze the right most column in the above diagram. By construction we have
\begin{equation}\label{f:decomp1}
   \Ext^{i,fin}_{\Mod(U)}(\mathbf{X}_U,V) = \bigoplus_{h \in U \backslash G / U} \Ext^i(\ind_U^{UhU}(k), V) \ .
\end{equation}
By the computation in the proof of Remark \ref{rem:J'-fin} the involution $\anti'$ respects this decomposition meaning that it maps the summand $\Ext^i(\ind_U^{UhU}(k), V)$ to the summand $\Ext^i(\ind_U^{Uh^{-1}U}(k), V)$. Furthermore, defining $U_h := U \cap hUh^{-1}$, we have the evaluation map $\ev_h : \ind_U^{UhU}(k) \rightarrow k$ sending $\phi$ to $\phi(h)$. By Frobenius reciprocity it induces the isomorphism
\begin{align*}
  \Fr_h : \Hom_{\Mod(U_h)}(k,V) = V^{U_h} & \xrightarrow{\;\cong\;} \Hom_{\Mod(U)}(\ind_U^{UhU}(k), V)   \\
                                        v & \longmapsto \Fr_h(v)(\phi) := \sum_{u \in U/U_h} \phi(uh) u v \ .
\end{align*}
Using an injective resolution of $V$ it gives rise to isomorphisms
\begin{equation*}
   \Fr_h : \Ext^i_{\Mod(U_h)}(k,V) \xrightarrow{\;\cong\;} \Ext^i_{\Mod(U)}(\ind_U^{UhU}(k), V)
\end{equation*}
for any $i \geq 0$

\begin{lemma}\label{lem:factorFR}
For any $i \geq 0$ and any $h \in G$, the diagram of isomorphisms
\begin{equation*}
  \xymatrix{
     \Ext^i_{\Mod(U)}(\ind_U^{UhU}(k),V)  \ar[r]^-{\anti'}_-{\cong} &  \Ext^i_{\Mod(U)}(\ind_U^{Uh^{-1}U}(k),V)  \\
    \Ext^i_{\Mod(U_h)}(k,V) \ar[r]^-{(h^{-1})_*}_-{\cong} \ar[u]^{\Fr_h}_{\cong} & \Ext^i_{\Mod(U_{h^{-1}})}(k,V) \ar[u]_{\Fr_{h^{-1}}}^{\cong}.   }
\end{equation*}
is commutative.
\end{lemma}
\begin{proof}
It suffices to establish the case $i = 0$, i.e., the commutativity of the diagram
\begin{equation*}
  \xymatrix{
     \Hom_{\Mod(U)}(\ind_U^{UhU}(k),V)  \ar[r]^-{\anti'} &  \Hom_{\Mod(U)}(\ind_U^{Uh^{-1}U}(k),V)  \\
               V^{U_h} \ar[r]^-{h^{-1} \cdot} \ar[u]^{\Fr_h} & V^{U_{h^{-1}}} \ar[u]_{\Fr_{h^{-1}}}.   }
\end{equation*}
In \eqref{f:J'-fin} we have seen that, for $F \in \Hom_{\Mod(U)}(\ind_U^{UhU}(k),V)$ and $\phi \in \ind_U^{Uh^{-1}U}(k)$, we have
\begin{equation*}
  \anti'(F)(\phi) = \sum_{g \in Uh^{-1}U/U} g(F(\chara_{g^{-1},U}^{g\phi(g)})) = \sum_{u \in U/U_{h^{-1}}} uh^{-1}(F(\chara_{h,U}^{\phi(uh^{-1})})) \ .
\end{equation*}
Note that $g\phi(g)=\phi(g)$ in this situation. Consider now $F = \Fr_h(v)$. Then
\begin{equation*}
  \Fr_h(v)(\chara_{h,U}^{\phi(uh^{-1})}) = \sum_{u' \in U/U_h} \chara_{h,U}^{\phi(u h^{-1})}(u'h) u' v = \phi(uh^{-1}) v \ .
\end{equation*}
It follows that
\begin{equation*}
   \anti'(\Fr_h(v))(\phi) = \sum_{u \in U/U_{h^{-1}}}  uh^{-1} \phi(uh^{-1}) v = \Fr_{h^{-1}}(h^{-1}v)(\phi) \ .
\end{equation*}
\end{proof}

\begin{lemma}\label{lem:Tr-cores}
   For any $i \geq 0$ and any $h \in G$, the diagram
\begin{equation*}
  \xymatrix{
     \Ext^i_{\Mod(U)}(\ind_U^{UhU}(k),V) \ar[r]^-{\Tr} &  \Ext^i_{\Mod(U)}(k,V) = H^i(U,V)      \\
     \Ext^i_{\Mod(U_h)}(k,V) = H^i(U_h,V)  \ar[u]^{\Fr_h}_{\cong} \ar[ur]_{\cores^{U_h}_U}                     }
\end{equation*}
is commutative.
\end{lemma}
\begin{proof}
Again it suffices to establish the case $i = 0$, i.e., the commutativity of the diagram
\begin{equation*}
  \xymatrix{
     \Hom_{\Mod(U)}(\ind_U^{UhU}(k),V) \ar[r]^-{\Tr} &  V^U      \\
                V^{U_h}  \ar[u]^{\Fr_h}_{\cong} \ar[ur]_{\cores^{U_h}_U}      \ .               }
\end{equation*}
We compute
\begin{align*}
    \Tr(\Fr_h(v))(1) & = \sum_{g \in UhU/U} \Fr_h(v)(\chara_{g,U}^1) = \sum_{g \in UhU/U} \sum_{u \in U/U_h} \chara_{g,U}^1(uh)uv   \\
                     & = \sum_{u \in U/U_h} uv = \cores^{U_h}_U(v) \ .
\end{align*}
\end{proof}

We further specialize the case under consideration to the top degree $i+j = d$ and the trivial representation $V = k$. By the above two lemmas the right most column in the diagram of Cor.\ \ref{cor:J'-J-Ext} can be rewritten as
\begin{equation*}
  \xymatrix{
    \oplus_{h \in U\backslash G / U} H^d(U_h,k) \ar[d]_{\sum_{h \in U\backslash G / U} \cores^{U_h}_U} \ar[rr]^{\oplus_h (h^{-1})_*} && \oplus_{h \in U\backslash G / U} H^d(U_{h^{-1}},k) \ar[d]^{\sum_{h \in U\backslash G / U} \cores^{U_{h^{-1}}}_U} \\
    H^d(U,k)  &&  H^d(U,k) .  }
\end{equation*}
From now on we \textbf{assume again} that $U$ is a Poincar\'e group. Then all the individual corestriction maps $\cores^{U_h}_U :  H^d(U_h,k) \xrightarrow{\cong} H^d(U,k)$ are isomorphisms of $1$-dimensional $k$-vector spaces. Moreover we have the commutative diagrams
\begin{equation}\label{diag:cores-chi}
  \xymatrix{
    H^d(U_h,k) \ar[d]_{\cores^{U_h}_U}^{\cong} \ar[rr]^{ (h^{-1})_*} &&  H^d(U_{h^{-1}},k) \ar[d]^{\cores^{U_{h^{-1}}}_U}_{\cong} \\
    H^d(U,k)  \ar[rr]^{\chi_G(h^{-1}) \cdot}  &&  H^d(U,k)   }
\end{equation}
where $\chi_G : G \rightarrow k^\times$ is the duality character of \cite{SS} Lemma 2.6. Note that $\chi_G$ is trivial on any open pro-$p$ subgroup of $G$.

This motivates us to twist the anti-involution $\anti$ on $\Ext^j_{\Mod(U)}(k, \mathbf{X}_U)$ as follows: Since $\chi_G$ is $U$-bi-invariant we may view $\chi_G$ as well as $\chi_G^{-1}$ as functions on $G/U$. Multiplying a function in $\mathbf{X}_U$ by $\chi_G^{\pm 1}$ then is a $U$-equivariant operator and therefore induces by functoriality a multiplication operator $\chi_G^{\pm 1} \cdot$ on $\Ext^j_{\Mod(U)}(k, \mathbf{X}_U)$. We define
\begin{equation*}
  \anti \otimes \chi_G := (\chi_G^{-1} \cdot) \circ \anti \ .
\end{equation*}
We have the decomposition
\begin{equation}\label{f:decomp2}
  \Ext^j_{\Mod(U)}(k, \mathbf{X}_U) = H^j(U, \mathbf{X}_U) = \bigoplus_{h \in U \backslash G / U} H^j(U,\ind_U^{UhU}(k)) \ .
\end{equation}
By Remark \ref{rem:anti-inv-support}, $\anti$ maps the summand $H^j(U,\ind_U^{UhU}(k))$ to the summand $H^j(U,\ind_U^{Uh^{-1}U}(k))$. Hence $(\anti \otimes \chi_G)(-) = \chi_G(h^{-1}) \cdot \anti(-)$.

\begin{remark}\label{rem:Jchi-anti}
   $\anti \otimes \chi_G$ is also an anti-involution. Since $\chi_G$ is $U$-bi-invariant this follows by using the formula for the Yoneda product given in \cite{OS} Prop.\ 5.3.
\end{remark}

In the special situation under consideration we now obtain the following strengthening of Cor.\ \ref{cor:J'-J-Ext}.

\begin{proposition}\label{prop:J'-Jchi-d}
   For any $i \geq 0$ the diagram of Yoneda pairings
\begin{equation*}
  \xymatrix{
    \Ext^i_{\Mod(U)}(\mathbf{X}_U,k)  \ar@{}[r]|{\times}  & \Ext^{d-i}_{\Mod(U)}(k, \mathbf{X}_U) \ar[d]_{\anti \otimes \chi_G}  \ar[r]^-{\circ} & H^d(U,k) \ar@{=}[d] \\
    \Ext^i_{\Mod(U)}(\mathbf{X}_U,k) \ar[u]_{\anti'}  \ar@{}[r]|{\times}  & \Ext^{d-i}_{\Mod(U)}(k, \mathbf{X}_U)  \ar[r]^-{\circ} & H^d(U,k)
     }
\end{equation*}
   is commutative.
\end{proposition}
\begin{proof}
It suffices to show that the diagram
\begin{equation*}
  \xymatrix{
    \Ext^i_{\Mod(U)}(\mathbf{X}_U,k) \ar@{}[r]|{\times}  & \Ext^{d-i}_{\Mod(U)}(k, \mathbf{X}_U) \ar[d]_{\anti \otimes \chi_G} \ar[r]^-{\langle\ ,\  \rangle} & \Ext^{d,fin}_{\Mod(U)}(\mathbf{X}_U,k)
    \\
    \Ext^i_{\Mod(U)}(\mathbf{X}_U,k) \ar[u]_{\anti'} \ar@{}[r]|{\times} & \Ext^{d-i}_{\Mod(U)}(k, \mathbf{X}_U) \ar[r]^-{\langle\ ,\  \rangle} & \Ext^{d,fin}_{\Mod(U)}(\mathbf{X}_U,k)
    \ar@{=}[u]
     }
\end{equation*}
is commutative. We do know from Cor.\ \ref{cor:J'-J-Ext} that
\begin{equation*}
  \xymatrix{
    \Ext^i_{\Mod(U)}(\mathbf{X}_U,k) \ar@{}[r]|{\times}  & \Ext^{d-i}_{\Mod(U)}(k, \mathbf{X}_U) \ar[d]_{\anti} \ar[r]^-{\langle\ ,\  \rangle} & \Ext^{d,fin}_{\Mod(U)}(\mathbf{X}_U,k)
    \\
    \Ext^i_{\Mod(U)}(\mathbf{X}_U,k) \ar[u]_{\anti'} \ar@{}[r]|{\times} & \Ext^{d-i}_{\Mod(U)}(k, \mathbf{X}_U) \ar[r]^-{\langle\ ,\  \rangle} & \Ext^{d,fin}_{\Mod(U)}(\mathbf{X}_U,k)
    \ar[u]^{\anti'}
     }
\end{equation*}
is commutative. In view of the decompositions \eqref{f:decomp1} and \eqref{f:decomp2} it further suffices to show, for any $h \in G$, the commutativity of
\begin{equation*}
  \xymatrix{
    \Ext^i_{\Mod(U)}(\mathbf{X}_U,k) & \times  & \Ext^{d-i}_{\Mod(U)}(k,\ind_U^{UhU}(k)) \ar[dd]_{\anti \otimes \chi_G}^{= \chi_G(h^{-1}) \cdot \anti} \ar[r]^-{\langle\ ,\  \rangle} & \Ext^d_{\Mod(U)}(\ind_U^{UhU}(k),k)  \ar[d]_{\Tr}^{\cong}
    \\
     &&& H^d(U,k) \\
    \Ext^i_{\Mod(U)}(\mathbf{X}_U,k) \ar[uu]_{\anti'} & \times & \Ext^{d-i}_{\Mod(U)}(k, \ind_U^{Uh^{-1}U}(k)) \ar[r]^-{\langle\ ,\  \rangle} & \Ext^d_{\Mod(U)}(\ind_U^{Uh^{-1}U}(k),k)
    \ar[u]^{\Tr}_{\cong}
     }
\end{equation*}
based on the fact that we know
\begin{equation*}
  \xymatrix{
    \Ext^i_{\Mod(U)}(\mathbf{X}_U,k) & \times  & \Ext^{d-i}_{\Mod(U)}(k,\ind_U^{UhU}(k)) \ar[d]_{\anti} \ar[r]^-{\langle\ ,\  \rangle} & \Ext^d_{\Mod(U)}(\ind_U^{UhU}(k),k)
    \\
    \Ext^i_{\Mod(U)}(\mathbf{X}_U,k) \ar[u]_{\anti'} & \times & \Ext^{d-i}_{\Mod(U)}(k, \ind_U^{Uh^{-1}U}(k)) \ar[r]^-{\langle\ ,\  \rangle} & \Ext^d_{\Mod(U)}(\ind_U^{Uh^{-1}U}(k),k)
    \ar[u]^{\anti'}
     }
\end{equation*}
to be commutative. Note that by the proof of Lemma \ref{lem:fin} all the four pairings are indeed subpairings of the previous ones. It remains to observe that by Lemmas \ref{lem:factorFR} and \ref{lem:Tr-cores} and the diagram \eqref{diag:cores-chi} we have the commutative diagram
\begin{equation*}
  \xymatrix{
    \Ext^d_{\Mod(U)}(\ind_U^{UhU}(k),k)  \ar[r]^-{\Tr} & H^d(U,k)  \\
    \Ext^d_{\Mod(U)}(\ind_U^{Uh^{-1}U}(k),k) \ar[u]^{\anti'} \ar[r]^-{\Tr} & H^d(U,k) \ar[u]_{\chi_G(h) \cdot} . }
\end{equation*}
\end{proof}

By using Frobenius reciprocity and Lemma \ref{lem:anti'-swap} we may rewrite Prop.\ \ref{prop:J'-Jchi-d} as a commutative diagram for the Yoneda pairing \eqref{f:Y-pairing}.

\begin{corollary}\label{cor:swap-Jchi}
   For any $i \geq 0$ the diagram of Yoneda pairings
\begin{equation*}
  \xymatrix{
    \Ext^i_{\Mod(G)}(\mathbf{X}_U \otimes_k \mathbf{X}_U,k) \qquad  \ar@{}[r]|{\times}  & \Ext^{d-i}_{\Mod(G)}(\mathbf{X}_U, \mathbf{X}_U) \ar[d]_{\anti \otimes \chi_G}  \ar[r]^-{\eqref{f:Y-pairing}} & H^d(U,k) \ar@{=}[d] \\
    \Ext^i_{\Mod(G)}(\mathbf{X}_U \otimes_k \mathbf{X}_U,k) \qquad \ar[u]_{\varsigma^*}  \ar@{}[r]|{\times}  & \Ext^{d-i}_{\Mod(G)}(\mathbf{X}_U, \mathbf{X}_U)  \ar[r]^-{\eqref{f:Y-pairing}} & H^d(U,k)
     }
\end{equation*}
   is commutative.
\end{corollary}


\subsection{The $E^*$-equivariance of the Yoneda pairing: Proof of equation (\ref{f:equiv2})}

The $U$-equivariant map $\mathbf{X}_U\longrightarrow \mathbf{X}_U \otimes_k \mathbf{X}_U$ sending $\varphi \mapsto \chara_U \otimes \varphi$ induces an isomorphism
\begin{equation}\label{XOX}
\Ext_{\Mod(G)}^*(\mathbf{X}_U \otimes_k \mathbf{X}_U,k)\overset{\sim}{\longrightarrow} \Ext_{\Mod(U)}^*(\mathbf{X}_{U},k)
\end{equation}
under which the $\mathbf{X}_U$-factor swap $\varsigma^*$ on the left-hand side corresponds to $\anti'$ on the right (see Lemma \ref{lem:anti'-swap}).

As we have just seen in Corollary \ref{cor:swap-Jchi} the Yoneda pairing (\ref{f:Y-pairing}) satisfies
$$
\langle \varsigma^*(f),\tau \rangle=\langle f, (\anti\otimes \chi_G)(\tau)\rangle.
$$
We consider the Yoneda algebra $E^*$ as a bimodule over itself. Observe that $\Ext_{\Mod(G)}^*(\mathbf{X}_U \otimes_k \mathbf{X}_U,k)$ has two commuting right $E^*$-module structures, one coming from each $\mathbf{X}_U$-factor. The pairing $\langle -,- \rangle$ has the following linearity properties for these actions, which summarize
(\ref{f:equiv1}) and (\ref{f:equiv2}).

\begin{proposition}\label{linear}
For any $\tau \in E^s$ we have the identities:
\begin{itemize}
\item[i.] $\langle f \cdot (1\otimes \tau),e\rangle=\langle f, \tau \cdot e\rangle$;
\item[ii.] $\langle f \cdot (\tau\otimes 1),e\rangle=(-1)^{s(d-i-s)}\cdot \langle f, e \cdot (\anti\otimes \chi_G)(\tau) \rangle$;
\end{itemize}
for all $f \in \Ext_{\Mod(G)}^i(\mathbf{X}_U \otimes_k \mathbf{X}_U,k)$ and $e \in E^{d-i-s}$.
\end{proposition}

\begin{proof}
The first identity in \ref{linear} is almost immediate from the definitions:
\begin{align*}
\langle f \cdot (1\otimes \tau),e\rangle &= f \circ (1\otimes \tau) \circ (1 \otimes e) \\
&= f \circ (1\otimes \tau \circ e) \\
&=\langle f, \tau \cdot e\rangle.
\end{align*}
For legibility we abuse notation here and suppress various shifts. For instance $\tau$ is a morphism
$\mathbf{X}_U \longrightarrow \mathbf{X}_U[s]$ in $D(G)$, and similarly for $e$, so $\tau \circ e$ really means
the composition $\tau[d-i-s]\circ e$ in $D(G)$. We will adopt this convention of suppressing shifts when they are obvious from the context.

To verify the second identity we relate the two $E^*$-module structures via the swap $\varsigma$. Note that
\begin{align*}
\varsigma^* \big(f\cdot (1\otimes \tau)\big)&=f \circ (1\otimes \tau) \circ \varsigma \\
&= f \circ \varsigma \circ (\tau \otimes 1)\\
&= \varsigma^*(f) \cdot  (\tau \otimes 1).
\end{align*}
As we observed in Remark \ref{rem:Jchi-anti} the twist $\anti\otimes \chi_G$ is an anti-involution of $E^*$. Therefore,
after introducing the auxiliary $\tilde{f}=\varsigma^*(f)$ the previous computation gives
\begin{align*}
\langle f \cdot (\tau\otimes 1),e\rangle&=f \circ (\tau \otimes 1) \otimes (1\otimes e) \\
&=\varsigma^* \big(\tilde{f}\cdot (1\otimes \tau)\big) \circ (1 \otimes e) \\
&=\langle \varsigma^* \big(\tilde{f}\cdot (1\otimes \tau)\big) , e \rangle \\
&=\langle \tilde{f}\cdot (1\otimes \tau) , (\anti\otimes \chi_G)(e) \rangle \\
&=\langle \tilde{f}, \tau\cdot (\anti\otimes \chi_G)(e) \rangle \\
&=\langle f, (\anti\otimes \chi_G)\big(\tau\cdot (\anti\otimes \chi_G)(e)\big) \rangle \\
&=(-1)^{\deg(e)\deg(\tau)}\cdot \langle f, e\cdot (\anti\otimes \chi_G)(\tau) \rangle \\
\end{align*}
In the fourth and sixth equalities we used Corollary \ref{cor:swap-Jchi}.
\end{proof}

\subsection{Applications to $\Delta$}\label{subsec:Delta}

Let $M=\bigoplus_{i \in \Bbb{Z}} M^i$ denote a graded {\it{right}} module over the Yoneda algebra $E^*$. In this section we will first construct a graded analogue of the functor $\Delta$ which takes $M$ to a dual graded right $E^*$-module $\Delta_{\gr}(M)$.

First consider the graded dual vector space $\Hom_k^{\gr}(M,k)$ whose $i^{\text{th}}$ graded piece is the space of linear forms $f: M \rightarrow k$ of degree $i$, in other words
linear forms $f: M^{-i}\rightarrow k$. Then take the shifted graded dual $\Hom_k^{\gr}(M,k)[-d]$. View $M$ as a left $E^*$-module via $\anti\otimes \chi_G$, which involves a sign,
$$
\tau \star e:=(-1)^{\deg(\tau)\deg(e)}e\cdot (\anti\otimes \chi_G)(\tau).
$$
We regard $E^*$ as a dg algebra with zero-differentials and follow \cite[10.6.3]{BL}. $\Hom_k^{\gr}(M,k)[-d]$ then acquires a natural
right $E^*$-module structure:

\begin{definition}
Let $\Delta_{\gr}(M)=\Hom_k^{\gr}(M,k)[-d]$ with the right $E^*$-module structure given by
$$
(f\tau)(e)=f(\tau \star e)=(-1)^{\deg(\tau)\deg(e)} f\big(e\cdot (\anti\otimes \chi_G)(\tau)\big)
$$
for homogeneous $\tau \in E^*$ and $e \in M$.
\end{definition}

Specializing to $E^*$ thought of as a right $E^*$-module over itself via multiplication, part two of Proposition \ref{linear} tells us the following.

\begin{corollary}\label{deltaE}
$\Ext_{\Mod(G)}^*(\mathbf{X}_U \otimes_k \mathbf{X}_U,k)\overset{\sim}{\longrightarrow} \Delta_{\gr}(E^*)$.
\end{corollary}

\begin{proof}
This is immediate. Recall that $E^*$ acts on the source via the first $\mathbf{X}_U$-factor.
\end{proof}

\begin{remark}\label{secondfact}
Part one of Proposition \ref{linear} shows the $E^*$-action via the second $\mathbf{X}_U$-factor intertwines with the right $E^*$-module structure on
$\Hom_k^{\gr}(E^*,k)[-d]$ arising from viewing $E^*$ as a {\it{left}} module over itself. Proposition \ref{prop:onesided-bimod} is a stronger version of this remark. We will often think of
$\Delta_{\gr}(E^*)$ as a right $E^*\otimes_k E^*$-module this way.
\end{remark}

Let $\Mod^{\gr}(E^*)$ denote the graded category of graded right $E^*$-modules. Similarly $\Mod^{\gr}(k)$ denotes the graded category of graded $k$-vector spaces. Thus
$\Hom_{\Mod^{\gr}(k)}(M,k)$ is the graded dual from before. The following standard properties are readily checked:
\begin{itemize}
\item $\Hom_{\Mod^{\gr}(E^*)}(X,Y[1])=\Hom_{\Mod^{\gr}(E^*)}(X,Y)[1]$;
\item $\Hom_{\Mod^{\gr}(k)}(Y \otimes_{E^*} X,Z)\simeq \Hom_{\Mod^{\gr}(E^*)}(Y, \Hom_{\Mod^{\gr}(k)}(X,Z)).$
\end{itemize}
In particular
\begin{align*}
\Hom_{\Mod^{\gr}(E^*)}(M, \Delta_{\gr}(E^*))&=\Hom_{\Mod^{\gr}(E^*)}(M, \Hom_{\Mod^{\gr}(k)}(E^*,k)[-d]) \\
&=\Hom_{\Mod^{\gr}(E^*)}(M, \Hom_{\Mod^{\gr}(k)}(E^*,k))[-d] \\
&\simeq \Hom_{\Mod^{\gr}(k)}(M \otimes_{E^*} E^*,k)[-d] \\
&\simeq \Hom_{\Mod^{\gr}(k)}(M,k)[-d] \\
&=\Delta_{\gr}(M).
\end{align*}
Since $M \rightsquigarrow \Delta_{\gr}(M)$ is clearly an exact functor this shows $\Delta_{\gr}(E^*)$ is an injective object of $\Mod^{\gr}(E^*)$ when endowed with the action in Remark \ref{secondfact}.

\begin{theorem}\label{cohdelta}
Let $M^{\bullet}$ be a left differential graded $H_U^{\bullet}$-module. Then there is a canonical isomorphism of
graded right $E^*$-modules
$$
h^*\big(\Delta(M^{\bullet})\big) \overset{\sim}{\longrightarrow} \Delta_{\gr}\big(h^*(M^{\bullet})\big).
$$
\end{theorem}

\begin{proof}
We invoke the Eilenberg-Moore spectral sequence \cite[Thm.~4.7]{KM} in our setup, using Corollary \ref{deltaE}:
$$
E_2^{s,t}=\Ext_{\Mod^{\gr}(E^*)}^{s,t}(h^*(M^{\bullet}), \Delta_{\gr}(E^*)) \Longrightarrow h^{s+t}(\Delta(M^{\bullet})).
$$
As $\Delta_{\gr}(E^*)$ is injective, $E_2^{s,t}=0$ when $s>0$. It therefore collapses on the $E_2$-page, so the following edge map is an isomorphism
$$
h^t(\Delta(M^{\bullet})) \overset{\sim}{\longrightarrow} \gr^0  h^t(\Delta(M^{\bullet}))\simeq E_{\infty}^{0,t} \simeq \Hom_{\Mod^{\gr}(E^*)}(h^*(M^{\bullet}), \Delta_{\gr}(E^*))^t  \overset{\sim}{\longrightarrow}
\Delta_{\gr}(h^*(M^{\bullet}))^t.
$$
Summing over $t$ yields the isomorphism in \ref{cohdelta}. A close inspection of the proof of \cite[Thm.~4.7]{KM} verifies the Eilenberg-Moore spectral sequence
for $R\Hom_{H_U^{\bullet}}(M^{\bullet},N^{\bullet})$ is functorial in $N^{\bullet}$ (a homomorphism of $H_U^{\bullet}$-modules $N_1^{\bullet}\rightarrow N_2^{\bullet}$ gives a morphism between the corresponding spectral sequences). Consequently the total edge map in \ref{cohdelta} is $E^*$-linear.
\end{proof}

Here is another interpretation and proof of \ref{cohdelta}. Let $V^{\bullet}$ be an object of $D(G)$ and consider its dual $R\underline{\Hom}(V^{\bullet},k)$. There is a natural pairing
$$
\langle\!\langle-,- \rangle \!\rangle: \Ext_{\Mod(G)}^*(\mathbf{X}_U,R\underline{\Hom}(V^{\bullet},k)) \times \Ext_{\Mod(G)}^*(\mathbf{X}_U, V^{\bullet}) \longrightarrow \Ext_{\Mod(G)}^*(\mathbf{X}_U \otimes_k \mathbf{X}_U,k)
$$
defined as follows. Start with morphisms $ \mathbf{X}_U\overset{\alpha}{\rightarrow} R\underline{\Hom}(V^{\bullet},k)[i]$ and $\mathbf{X}_U \overset{\beta}{\rightarrow} V^{\bullet}[j]$ in
$D(G)$. Then $\langle\!\langle \alpha, \beta \rangle \!\rangle$ is defined as $\alpha \otimes \beta$ composed with the adjunction counit
$R\underline{\Hom}(V^{\bullet},k) \otimes V^{\bullet} \overset{c}{\rightarrow} k$. More precisely $\langle\!\langle \alpha, \beta \rangle \!\rangle=c[i+j]\circ(\alpha \otimes \beta)$. Clearly
$\langle\!\langle-,- \rangle \!\rangle$ is $E^* \otimes_k E^*$-linear. Composition with the isomorphism in Corollary \ref{deltaE} therefore gives an $E^*$-linear map
\begin{align*}
\zeta_{V^{\bullet}}: \Ext_{\Mod(G)}^*(\mathbf{X}_U,R\underline{\Hom}(V^{\bullet},k))  & \longrightarrow  \Hom_{\Mod^{\gr}(E^*)}(\Ext_{\Mod(G)}^*(\mathbf{X}_U, V^{\bullet}), \Delta_{\gr}(E^*)) \\
& \overset{\sim}{\longrightarrow} \Delta_{\gr}(\Ext_{\Mod(G)}^*(\mathbf{X}_U, V^{\bullet})).
\end{align*}
This is an isomorphism as we will show next.

\begin{theorem}\label{maindual}
Let $V^{\bullet}$ be a complex of smooth $G$-representations over $k$. Then there is a canonical isomorphism of graded right $E^*$-modules
$$
H^*(U, R\underline{\Hom}(V^{\bullet},k)) \overset{\sim}{\longrightarrow} \Delta_{\gr}(H^*(U,V^{\bullet})).
$$
\end{theorem}

\begin{proof}
We constructed the $E^*$-linear map $\zeta_{V^{\bullet}}$ above. For the purpose of this proof, let $\mathcal{C}$ denote the (strictly) full subcategory of $D(G)$ whose objects are the
$V^{\bullet}$ for which $\zeta_{V^{\bullet}}$ is an isomorphism. That is, those $V^{\bullet}$ for which
$$
\zeta_{V^{\bullet}}^i: \Ext_{\Mod(G)}^i(\mathbf{X}_U,R\underline{\Hom}(V^{\bullet},k)) \longrightarrow \Hom_k(\Ext_{\Mod(G)}^{d-i}(\mathbf{X}_U, V^{\bullet}),k)
$$
is bijective for all $i \in \Bbb{Z}$. By Corollary \ref{deltaE} we know $\mathbf{X}_U \in \text{Ob}(\mathcal{C})$. Essentially by the five lemma $\mathcal{C}$ is a triangulated subcategory. It is closed under direct sums. Indeed $\mathbf{X}_U$ is a compact object of $D(G)$ by \cite[Lem.~4]{DGA} so $\Hom_k(\Ext_{\Mod(G)}^{d-i}(\mathbf{X}_U, -),k)$ takes direct sums to direct products. The same is true for the functor $\Ext_{\Mod(G)}^i(\mathbf{X}_U,R\underline{\Hom}(-,k))$ by adjunction:
$$
\Ext_{\Mod(G)}^i(\mathbf{X}_U,R\underline{\Hom}(V^{\bullet},k))\simeq \Ext_{\Mod(G)}^i(\mathbf{X}_U \otimes_k V^{\bullet},k).
$$
By \cite[Prop.~6]{DGA} we conclude that $\mathcal{C}=D(G)$.
\end{proof}

\section{Transfer of the tensor product}\label{sec:product}

We let $C(G)$ denote the Grothendieck abelian category of all unbounded complexes with terms in $\Mod(G)$. As noted in \cite[Sect.~3.1]{ScSc} there is a natural model structure
on $C(G)$ for which
\begin{itemize}
\item the cofibrations are the monomorphisms;
\item the weak equivalences are the quasi-isomorphisms;
\item the fibrations are the epimorphisms with homotopically injective and term-wise injective kernel.
\end{itemize}
We let $\underline{C}(G)$ be the dg-enhancement of $K(G)$. This is the dg-category whose objects are complexes $V^{\bullet}$ and whose morphisms are
$\Hom_{\Mod(G)}^{\bullet}(V^{\bullet},V'^{\bullet})$. The tensor product of complexes can be enriched to a dg-bifunctor $\otimes_k$ on $\underline{C}(G)$ in the sense that
$$
\Hom_{\Mod(G)}^{\bullet}(V^{\bullet},V'^{\bullet}) \otimes_k \Hom_{\Mod(G)}^{\bullet}(W^{\bullet},W'^{\bullet}) \longrightarrow
\Hom_{\Mod(G)}^{\bullet}(V^{\bullet}\otimes_k W^{\bullet},V'^{\bullet}\otimes_k W'^{\bullet})
$$
is a morphism of complexes. In $\underline{C}(G)$ we consider the full dg-subcategory $\underline{I}(G)$ consisting of all fibrant objects (that is, the homotopically injective and term-wise injective complexes). The homotopy category of $\underline{I}(G)$ is naturally equivalent to $D(G)$ via the restriction of $q_G$.

Unless $G$ is compact $\mathcal{I}^{\bullet} \otimes_k \mathcal{I}^{\bullet}$ need not be an injective resolution of $\mathbf{X}_U \otimes_k \mathbf{X}_U$. In order to resolve
$\mathbf{X}_U \otimes_k \mathbf{X}_U$ in a dg-functorial way we make use of the following special case of a theorem of Schn\"u{}rer.

\begin{theorem}\label{enhance}
There exists a dg-functor $\underline{i}: \underline{C}(G) \longrightarrow \underline{I}(G)$ along with a natural transformation
$\phi: \text{Id}\rightarrow (\text{inclusion}) \circ \underline{i}$ such that the morphism $\phi_{V^{\bullet}}: V^{\bullet} \rightarrow \underline{i} (V^{\bullet})$ is a trivial cofibration for all
complexes $V^{\bullet}$. (A trivial cofibration is a monomorphic quasi-isomorphism.)
\end{theorem}

\begin{proof}
This is \cite[Thm.~3.2]{ScSc} which is deduced from \cite[Thm.~4.3]{Schn}.
\end{proof}

Moreover, after passing to homotopy categories \cite[Lem.~3.3]{ScSc} tells us $[\underline{i}]$ factors uniquely through $q_G$ as in the diagram
\begin{equation*}
  \xymatrix{
    K(G) \ar@{=}[r] \ar[d]_{q_G}& [\underline{C}(G)] \ar[r]^{[\underline{i}]} &  [\underline{I}(G)]  \ar@{^{(}->}[r] & K_{\text{inj}}(G). \\
      D(G) \ar@{-->}[urr]  \ar[urrr]_{\mathbf{i}}&   &  & }
\end{equation*}

Once and for all we fix a pair $(\underline{i},\phi)$ as in Theorem \ref{enhance}. In particular this gives a fibrant replacement
$$
\phi_{\mathcal{I}^{\bullet} \otimes_k \mathcal{I}^{\bullet}}\circ (\kappa \otimes \kappa): \mathbf{X}_U \otimes_k \mathbf{X}_U \longrightarrow \mathcal{I}^{\bullet} \otimes_k \mathcal{I}^{\bullet} \longrightarrow
\underline{i}(\mathcal{I}^{\bullet} \otimes_k \mathcal{I}^{\bullet}).
$$
Furthermore, by dg-functoriality there is a natural homomorphism of differential graded algebras
$$
H_U^{\bullet}\otimes_k H_U^{\bullet} \longrightarrow \End_{\Mod(G)}^{\bullet}\big(\underline{i}(\mathcal{I}^{\bullet} \otimes_k \mathcal{I}^{\bullet})\big)^{\text{op}}.
$$
In other words $\underline{i}(\mathcal{I}^{\bullet} \otimes_k \mathcal{I}^{\bullet})$ is a right dg-module over $H_U^{\bullet}\otimes_k H_U^{\bullet}$. This applies verbatim to $n$-fold tensor products as well and leads to the next key definition.

\begin{definition}\label{dgbi}
For every non-negative integer $n$ we let
$$
H_U^{\bullet}(n)=\Hom_{\Mod(G)}^{\bullet}\big(\mathcal{I}^{\bullet},\underline{i}(\underbrace{\mathcal{I}^{\bullet} \otimes_k \cdots \otimes_k \mathcal{I}^{\bullet}}_{n})\big).
$$
This is a dg $(H_U^{\bullet}, \underbrace{H_U^{\bullet}\otimes_k \cdots \otimes_k H_U^{\bullet}}_{n})$-bimodule. By convention $H_U^{\bullet}(0)=
\Hom_{\Mod(G)}^{\bullet}\big(\mathcal{I}^{\bullet},\underline{i}(k)\big)$.
\end{definition}

\begin{remark}
Note that $q_H(H_U^{\bullet}(0))\simeq \mathbf{U}$ is the unit object in $D(H_U^{\bullet})$. Also, there is a canonical quasi-isomorphism of
dg $(H_U^{\bullet},H_U^{\bullet})$-bimodules $H_U^{\bullet}\overset{\sim}{\longrightarrow} H_U^{\bullet}(1)$ obtained by composition with $\phi_{\mathcal{I}^{\bullet}}$.
The next bimodule $H_U^{\bullet}(2)$ will play a crucial role in our description of the monoidal structure on $D(H_U^{\bullet})$.
\end{remark}

\begin{remark}\label{operad}
The cohomology of $H_U^{\bullet}(n)$ is the graded $(\underbrace{E^*\otimes_k \cdots \otimes_k E^*}_{n},E^*)$-bimodule
$$
h^*\big(H_U^{\bullet}(n)\big)=\Ext_{\Mod(G)}^*(\mathbf{X}_U, \underbrace{\mathbf{X}_U \otimes_k \cdots \otimes_k \mathbf{X}_U}_{n})=: E^*(n).
$$
Altogether they form an operad. See \cite[Rem.~2.3 (i)]{KM} for example.
\end{remark}

By \cite[Prop.~1.5]{DM} applied to $\underline{C}(G)$ there is a right $S_n$-action on $H_U^{\bullet}(n)$ essentially given by permuting the $\mathcal{I}^{\bullet}$-factors and keeping track of Koszul signs. Here we will only need this for $n=2$, in which case it comes down to the involution
$$
\varsigma_*: H_U^{\bullet}(2) \overset{\sim}{\longrightarrow} H_U^{\bullet}(2)
$$
induced by the swap $\varsigma: \mathcal{I}^{\bullet}\otimes_k \mathcal{I}^{\bullet} \overset{\sim}{\longrightarrow} \mathcal{I}^{\bullet}\otimes_k \mathcal{I}^{\bullet}$ (which involves a Koszul sign). Obviously $\varsigma_*$ is left $H_U^{\bullet}$-linear, and the two right $H_U^{\bullet}$-module structures on $H_U^{\bullet}(2)$ are intertwined by $\varsigma_*$.

The functor $q_H$ has a fully faithful left adjoint functor $\mathbf{p}$ with essential image $K_{\text{pro}}(H_U^{\bullet})$. It comes with natural transformations
\begin{itemize}
\item $\mathbf{p} \circ q_H \longrightarrow \text{Id}_{K(H_U^{\bullet})}$;
\item $\text{Id}_{D(H_U^{\bullet})} \overset{\sim}{\longrightarrow} q_H \circ \mathbf{p}$.
\end{itemize}
Note that $(\mathbf{p} \circ q_H)(X^{\bullet}) \longrightarrow X^{\bullet}$ is a quasi-isomorphism for all dg $H_U^{\bullet}$-modules $X^{\bullet}$. In particular the latter map is an isomorphism in $K(H_U^{\bullet})$ if $X^{\bullet}$ is homotopically projective.

\begin{definition}
Let $M^{\bullet}$ and $N^{\bullet}$ be objects of $D(H_U^{\bullet})$. Their tensor product is
$$
M^{\bullet} \boxtimes N^{\bullet}=q_H\bigg(H_U^{\bullet}(2) \otimes_{H_U^{\bullet}\otimes_k H_U^{\bullet}} (\mathbf{p}M^{\bullet}\otimes_k  \mathbf{p}N^{\bullet})\bigg).
$$
Analogously for morphisms. This defines a bifunctor $\boxtimes: D(H_U^{\bullet})\times D(H_U^{\bullet}) \longrightarrow D(H_U^{\bullet})$.
\end{definition}

Note that $\mathbf{p}M^{\bullet}\otimes_k  \mathbf{p}N^{\bullet}$ belongs to $K_{\text{pro}}(H_U^{\bullet}\otimes_k H_U^{\bullet})$ by tensor-hom adjunction.

\begin{proposition}
There are functorial isomorphisms $M^{\bullet} \boxtimes N^{\bullet} \overset{\sim}{\longrightarrow} N^{\bullet} \boxtimes M^{\bullet}$.
\end{proposition}

\begin{proof}
Start with the symmetry for complexes
\begin{align*}
  s : \mathbf{p}M^{\bullet}\otimes_k  \mathbf{p}N^{\bullet} &  \overset{\sim}{\longrightarrow} \mathbf{p}N^{\bullet}\otimes_k  \mathbf{p}M^{\bullet} \\
                 x\otimes y & \longmapsto (-1)^{\deg(x)\deg(y)}\cdot (y \otimes x) \ .
\end{align*}
One checks $\varsigma_*\otimes_k s$ factors through the tensor product over $H_U^{\bullet}\otimes_k H_U^{\bullet}$ by observing that
$$
\varsigma \circ (a \otimes b)=(-1)^{\deg(a)\deg(b)}\cdot (b \otimes a) \circ \varsigma
$$
for homogeneous $a,b \in H_U^{\bullet}$, where $a \otimes b$ and $b \otimes a$ denote the induced endomorphisms of $\mathcal{I}^{\bullet}\otimes_k \mathcal{I}^{\bullet}$ in
$\underline{C}(G)$. Indeed, upon applying $\underline{i}$ to the previous equation, we deduce that
$$
\varsigma_*(c(a \otimes b))=(-1)^{\deg(a)\deg(b)}\cdot \varsigma_*(c)(b \otimes a)
$$
for $c \in H_U^{\bullet}(2)$. The claim that $\varsigma_*\otimes_k s$ factors through $H_U^{\bullet}(2) \otimes_{H_U^{\bullet}\otimes_k H_U^{\bullet}} (\mathbf{p}M^{\bullet}\otimes_k  \mathbf{p}N^{\bullet})$ now follows from an easy computation which we omit.
\end{proof}

The only way we know how to show $\boxtimes$ is associative is to relate it to $\otimes_k$ on $D(G)$ via the equivalence $H$.

\begin{theorem}\label{tensor}
Let $V^{\bullet}$ and $W^{\bullet}$ be objects of $D(G)$. Then there is a natural isomorphism
$$
H(V^{\bullet}) \boxtimes H(W^{\bullet}) \overset{\sim}{\longrightarrow} H(V^{\bullet}\otimes_k W^{\bullet})
$$
in $D(H_U^{\bullet})$.
\end{theorem}

\begin{proof}
First, recall that $V^{\bullet}\otimes_k W^{\bullet}$ is defined as $q_G(\mathbf{i}V^{\bullet}\otimes_k \mathbf{i}W^{\bullet})$. Moreover, by the diagram after Theorem \ref{enhance}
we know that $\underline{i}(C^{\bullet})=(\mathbf{i}\circ q_G)(C^{\bullet})$ for all complexes $C^{\bullet}$ in $C(G)$. Applying this observation to the complex
$C^{\bullet}=\mathbf{i}V^{\bullet}\otimes_k \mathbf{i}W^{\bullet}$ we find that
\begin{align*}
H(V^{\bullet}\otimes_k W^{\bullet})  &=q_H\big(\Hom_{\Mod(G)}^{\bullet}(\mathcal{I}^{\bullet},\mathbf{i}(V^{\bullet}\otimes_k W^{\bullet}))\big) \\
&= q_H\big(\Hom_{\Mod(G)}^{\bullet}(\mathcal{I}^{\bullet},(\mathbf{i}\circ q_G)(\mathbf{i}V^{\bullet}\otimes_k \mathbf{i}W^{\bullet}))\big) \\
&=q_H\big(\Hom_{\Mod(G)}^{\bullet}(\mathcal{I}^{\bullet},\underline{i}(\mathbf{i}V^{\bullet}\otimes_k \mathbf{i}W^{\bullet}))\big).
\end{align*}
The adjunction counit $\mathbf{p} \circ q_H \longrightarrow \text{Id}_{K(H_U^{\bullet})}$ gives a map
\begin{align*}
H(V^{\bullet}) \boxtimes H(W^{\bullet})&=q_H\bigg(H_U^{\bullet}(2) \otimes_{H_U^{\bullet}\otimes_k H_U^{\bullet}} (\mathbf{p}H(V^{\bullet})\otimes_k  \mathbf{p}H(W^{\bullet}))\bigg) \\
&\longrightarrow q_H\bigg(H_U^{\bullet}(2) \otimes_{H_U^{\bullet}\otimes_k H_U^{\bullet}} \Hom_{\Mod(G)}^{\bullet}(\mathcal{I}^{\bullet}, \mathbf{i}V^{\bullet}) \otimes_k  \Hom_{\Mod(G)}^{\bullet}(\mathcal{I}^{\bullet}, \mathbf{i}W^{\bullet})\bigg)
\end{align*}
which we compose with the following map: Given a pair of morphisms $\alpha \in \Hom_{\Mod(G)}^{\bullet}(\mathcal{I}^{\bullet}, \mathbf{i}V^{\bullet})$ and
$\beta \in \Hom_{\Mod(G)}^{\bullet}(\mathcal{I}^{\bullet}, \mathbf{i}W^{\bullet})$ along with a $c \in H_U^{\bullet}(2)$ we form $\alpha \otimes \beta$ and precompose
$\underline{i}(\alpha \otimes \beta)$ with $c$ in the dg-category $\underline{C}(G)$. This gives a map
$$
q_H\bigg(H_U^{\bullet}(2) \otimes_{H_U^{\bullet}\otimes_k H_U^{\bullet}} \Hom_{\Mod(G)}^{\bullet}(\mathcal{I}^{\bullet}, \mathbf{i}V^{\bullet}) \otimes_k  \Hom_{\Mod(G)}^{\bullet}(\mathcal{I}^{\bullet}, \mathbf{i}W^{\bullet})\bigg)  \longrightarrow H(V^{\bullet}\otimes_k W^{\bullet})
$$
by sending $c \otimes \alpha \otimes \beta \mapsto \underline{i}(\alpha \otimes \beta) \circ c$. We need to show the resulting map
$$
\psi_{V^{\bullet},W^{\bullet}}: H(V^{\bullet}) \boxtimes H(W^{\bullet}) \longrightarrow H(V^{\bullet}\otimes_k W^{\bullet})
$$
is an isomorphism in $D(H_U^{\bullet})$. We verify this in increasing generality starting from the case $V^{\bullet}=W^{\bullet}=\mathbf{X}_U$. Let us take
$\mathcal{I}^{\bullet}=\mathbf{i}\mathbf{X}_U$ for simplicity (in general they are homotopy equivalent). Then
$$
\mathbf{p}H(\mathbf{X}_U)=(\mathbf{p} \circ q_H)\big(\Hom_{\Mod(G)}^{\bullet}(\mathcal{I}^{\bullet}, \mathbf{i}\mathbf{X}_U)\big)=(\mathbf{p} \circ q_H)(H_U^{\bullet})
\overset{\sim}{\longrightarrow} H_U^{\bullet}
$$
since $H_U^{\bullet}$ is obviously homotopically projective. Therefore $\psi_{\mathbf{X}_U, \mathbf{X}_U}$ is a composition of isomorphisms
$$
H(\mathbf{X}_U) \boxtimes H(\mathbf{X}_U) \overset{\sim}{\longrightarrow} q_H(H_U^{\bullet}(2)) \overset{\sim}{\longrightarrow} H(\mathbf{X}_U\otimes_k \mathbf{X}_U).
$$
As an intermediate step we introduce the strictly full subcategory $\mathcal{C}$ of $D(G)$ whose objects are the $V^{\bullet}$ for which $\psi_{V^{\bullet},\mathbf{X}_U}$ is an isomorphism. We have just verified that $\mathbf{X}_U$ is an object of $\mathcal{C}$. Since $H$ and $(-)\boxtimes H(\mathbf{X}_U)$ are triangulated functors the 2-out-of-3 property
\cite[Prop.~1.1.20]{Nee} shows $\mathcal{C}$ is a triangulated subcategory. By Lemma \ref{distribute} below $\mathcal{C}$ is furthermore closed under all coproducts. Therefore
$\mathcal{C}=D(G)$ by \cite[Prop.~6]{DGA}.

Finally fix an arbitrary $V^{\bullet}$ and consider the strictly full subcategory $\mathcal{D}_{V^{\bullet}}$ of $D(G)$ with objects $W^{\bullet}$ for which
$\psi_{V^{\bullet},W^{\bullet}}$ is an isomorphism. The above intermediate step shows $\mathbf{X}_U$ is an object of $\mathcal{D}_{V^{\bullet}}$, and the same reasoning allows us to conclude $\mathcal{D}_{V^{\bullet}}=D(G)$ as desired.
\end{proof}

The following observation was used in the previous proof.

\begin{lemma}\label{distribute}
The category $D(H_U^{\bullet})$ has arbitrary coproducts. Let $(M_i^{\bullet})_{i\in I}$ be a collection of objects from $D(H_U^{\bullet})$. Then
\begin{itemize}
\item[(i)] $\bigoplus_{i \in I}\mathbf{p}(M_i^{\bullet})\overset{\sim}{\longrightarrow}\mathbf{p}(\bigoplus_{i \in I} M_i^{\bullet})$;
\item[(ii)] There are functorial isomorphisms
$$
\bigoplus_{i \in I} (M_i^{\bullet}\boxtimes N^{\bullet}) \overset{\sim}{\longrightarrow} (\bigoplus_{i \in I} M_i^{\bullet})\boxtimes N^{\bullet};
$$
\item[(iii)] The functor $H$ and its quasi-inverse $T$ preserve arbitrary coproducts.
\end{itemize}
\end{lemma}

\begin{proof}
By \cite[Rem.~2]{DGA} the category $D(G)$ has arbitrary coproducts, and therefore so does the equivalent category $D(H_U^{\bullet})$.  As a left adjoint the functor $\mathbf{p}$ respects arbitrary coproducts so that we have $\bigoplus_{i \in I}\mathbf{p}(M_i^{\bullet})\overset{\sim}{\longrightarrow}\mathbf{p}(\bigoplus_{i \in I} M_i^{\bullet})$ where the source is a direct sum of complexes. Applying $q_H$ we deduce that $\bigoplus_{i \in I} M_i^{\bullet}$ is isomorphic to $q_H(\bigoplus_{i \in I}\mathbf{p}(M_i^{\bullet}))$. As a result
\begin{align*}
\bigoplus_{i \in I} (M_i^{\bullet}\boxtimes N^{\bullet}) &\simeq  q_H\bigg(\bigoplus_{i \in I} \mathbf{p}(M_i^{\bullet}\boxtimes N^{\bullet})\bigg) \\
&= q_H\bigg(\bigoplus_{i \in I} (\mathbf{p} \circ q_H)( H_U^{\bullet}(2) \otimes_{H_U^{\bullet}\otimes_k H_U^{\bullet}} (\mathbf{p}M_i^{\bullet}\boxtimes \mathbf{p}N^{\bullet}))\bigg)  \\
&\simeq q_H\bigg(\bigoplus_{i \in I} H_U^{\bullet}(2) \otimes_{H_U^{\bullet}\otimes_k H_U^{\bullet}} (\mathbf{p}M_i^{\bullet}\otimes_k  \mathbf{p}N^{\bullet}) \bigg) \\
&\simeq (\bigoplus_{i \in I} M_i^{\bullet})\boxtimes N^{\bullet}
\end{align*}
where the third isomorphism comes from the quasi-isomorphisms $(\mathbf{p} \circ q_H)(X^{\bullet}) \longrightarrow X^{\bullet}$ with
$X^{\bullet}=H_U^{\bullet}(2) \otimes_{H_U^{\bullet}\otimes_k H_U^{\bullet}} (\mathbf{p}M_i^{\bullet}\otimes_k  \mathbf{p}N^{\bullet})$. Part (iii) is obvious.
\end{proof}

Passing to cohomology in Theorem \ref{tensor} yields an Eilenberg-Moore spectral sequence:

\begin{corollary}\label{EMSS}
For any two complexes $V^{\bullet}$ and $W^{\bullet}$ of smooth $G$-representations over $k$ there is a convergent $H_U$-equivariant spectral sequence
$$
E_2^{s,t}=\Tor_{E^*\otimes_k E^*}^{s,t}\big(H^*(U,V^{\bullet})\otimes_k H^*(U,W^{\bullet}), E^*(2)\big) \Longrightarrow H^{s+t}(U,V^{\bullet}\otimes_k W^{\bullet}).
$$
\end{corollary}

\begin{proof}
Apply the first half of \cite[Thm.~4.7]{KM} to $A=H_U^{\bullet}\otimes_k H_U^{\bullet}$ and the dg-modules $M=H_U^{\bullet}(2)$ and
$N=H(V^{\bullet}) \otimes_k H(W^{\bullet})$ in the notation of loco citato. Then $M \otimes_A^L N$ is isomorphic to our
$H(V^{\bullet}) \boxtimes H(W^{\bullet})$, which by Theorem \ref{tensor} is $H(V^{\bullet}\otimes_k W^{\bullet})$. According to \cite[Def.~4.6]{KM}
the spectral sequence converges to $\Tor_A^*(M,N)=H^*(U,V^{\bullet}\otimes_k W^{\bullet})$. The $E_2$-page is given by
$$
E_2^{s,t}=\Tor_{h^*(A)}^{s,t}(h^*(M),h^*(N))=\Tor_{E^*\otimes_k E^*}^{s,t}(H^*(U,V^{\bullet})\otimes_k H^*(U,W^{\bullet}), E^*(2))
$$
as claimed. The order of the inputs is interchanged since $h^*(H_U^{\bullet})$ is the opposite of $E^*$. The notation $E^*(n)=h^*(H_U^{\bullet}(n))$ was introduced in Remark \ref{operad}. Note also that $s$ is the negative of the homological degree, so
on the initial page $E_2^{s,t}=0$ unless $s \leq 0$.

The (strong) convergence of the $\Tor$-sequence is mentioned right after \cite[Thm.~4.7]{KM}. It can be deduced from \cite[Ch.~XI, Prop.~3.2]{Mac}: The proof in \cite{KM} ultimately comes down to the choice of a semifree resolution, which by definition is bounded below and convergent above in the terminology of \cite{Mac}.
\end{proof}

Here is an obvious reformulation of Theorem \ref{tensor}. Start with $M^{\bullet}$ and $N^{\bullet}$ in $D(H_U^{\bullet})$. Then \ref{tensor} combined with
the natural transformation $\text{Id}_{D(H_U^{\bullet})} \overset{\sim}{\longrightarrow} H \circ T$ give isomorphisms
$$
M^{\bullet} \boxtimes N^{\bullet} \overset{\sim}{\longrightarrow} H(T(M^{\bullet}))\boxtimes H(T(N^{\bullet}))  \overset{\sim}{\longrightarrow} H(T(M^{\bullet})\otimes_k T(N^{\bullet})).
$$
Applying $T$ and invoking the transformation $T \circ H \overset{\sim}{\longrightarrow} \text{Id}_{D(G)}$ yields the isomorphism
$$
T(M^{\bullet} \boxtimes N^{\bullet} )\overset{\sim}{\longrightarrow} T(M^{\bullet})\otimes_k T(N^{\bullet}).
$$
The monoidal structure on $D(H_U^{\bullet})$ can be expressed in terms of these maps:

\begin{itemize}
\item \emph{Unit object}.
$$
M^{\bullet}\boxtimes \mathbf{U}\overset{\sim}{\longrightarrow}H(T(M^{\bullet})\otimes_k T(H(k)))\overset{\sim}{\longrightarrow}
H(T(M^{\bullet})\otimes_k k)\overset{\sim}{\longrightarrow} H(T(M^{\bullet}))\overset{\sim}{\longleftarrow} M^{\bullet}.
$$
\item \emph{Associativity}.
\begin{align*}
(M_1^{\bullet}\boxtimes M_2^{\bullet})\boxtimes M_3^{\bullet} &\overset{\sim}{\longrightarrow} H(T(M_1^{\bullet}\boxtimes M_2^{\bullet})\otimes_k T(M_3^{\bullet})) \\
&\overset{\sim}{\longrightarrow} H((T(M_1^{\bullet})\otimes_k T(M_2^{\bullet}))\otimes_k T(M_3^{\bullet})) \\
& \overset{\sim}{\longrightarrow} H(T(M_1^{\bullet})\otimes_k (T(M_2^{\bullet})\otimes_k T(M_3^{\bullet}))) \\
& \overset{\sim}{\longleftarrow} H(T(M_1^{\bullet})\otimes_k T(M_2^{\bullet}\boxtimes M_3^{\bullet})) \\
& \overset{\sim}{\longleftarrow} M_1^{\bullet}\boxtimes (M_2^{\bullet}\boxtimes M_3^{\bullet}).
\end{align*}
\end{itemize}
We do not have a description of these maps which does not go through $D(G)$ via $H$ and $T$.

\begin{remark}
Our Theorem \ref{tensor} can be strengthened to $n$ complexes $V_1^{\bullet},\ldots, V_n^{\bullet}$ with minor modifications of the proof. We extend $\boxtimes$ to an $n$-fold tensor product as in \cite[Prop.~1.5]{DM} for instance.  Then, if
$M_1^{\bullet},\ldots, M_n^{\bullet}$ are objects of $D(H_U^{\bullet})$, their tensor product can be expressed in terms of $H_U^{\bullet}(n)$ as
$$
M_1^{\bullet} \boxtimes \cdots \boxtimes M_n^{\bullet}\simeq
q_H\bigg(H_U^{\bullet}(n) \otimes_{H_U^{\bullet}\otimes_k \cdots \otimes_k H_U^{\bullet}} (\mathbf{p}M_1^{\bullet}\otimes_k  \cdots \otimes_k \mathbf{p}M_n^{\bullet})\bigg).
$$
\end{remark}

\section{The internal $\Hom$-functor for dg modules}

In this section we transfer $R\underline{\Hom}$ to $D(H_U^{\bullet})$.

\begin{definition}\label{internal}
Let $M^{\bullet}$ and $N^{\bullet}$ be objects of $D(H_U^{\bullet})$. Then
$$
\Hom^{\boxtimes}(M^{\bullet},N^{\bullet})=q_H \Hom_{H_U^{\bullet}}\bigg( \mathbf{p}\circ q_H \bigg[ H_U^{\bullet}(2)\otimes_{H_U^{\bullet}}\mathbf{p}M^{\bullet} \bigg], \mathbf{p}N^{\bullet} \bigg).
$$
\end{definition}

This definition takes a bit of unwinding:

\begin{itemize}
\item The tensor product $H_U^{\bullet}(2)\otimes_{H_U^{\bullet}}\mathbf{p}M$ is with respect to the right $H_U^{\bullet}$-module structure on $H_U^{\bullet}(2)$ arising from the \emph{second} factor in $H_U^{\bullet}\otimes_k H_U^{\bullet}$.
\item When forming $\Hom_{H_U^{\bullet}}$ (introduced in \cite[10.8]{BL}) we view $\mathbf{p}\circ q_H \big[\cdots\big]$ as a left $H_U^{\bullet}$-module through the left
action of $H_U^{\bullet}$ on $H_U^{\bullet}(2)$.
\item The space $\Hom_{H_U^{\bullet}}(\cdots)$ on the right-hand side in \ref{internal} becomes a left $H_U^{\bullet}$-module via the right action of $H_U^{\bullet}$ on
$H_U^{\bullet}(2)$ via the \emph{first} factor in $H_U^{\bullet}\otimes_k H_U^{\bullet}$.
\end{itemize}

\begin{remark}
At least morally $\boxtimes$ is the derived tensor product $H_U^{\bullet}(2) \otimes_{H_U^{\bullet}\otimes_k H_U^{\bullet} }^L M^{\bullet}\otimes_k N^{\bullet}$. Similarly we think of
$\Hom^{\boxtimes}$ as $R\Hom_{H_U^{\bullet}}(H_U^{\bullet}(2)\otimes_{H_U^{\bullet}}^LM^{\bullet}  ,N^{\bullet})$. We have adopted the notation $\boxtimes$ and
$\Hom^{\boxtimes}$ from \cite[Part V, Df.~1.1, Df.~1.2]{KM}.
\end{remark}

\begin{proposition}\label{adjoint}
There are functorial isomorphisms
$$
\Hom_{D(H_U^{\bullet})}(M^{\bullet} \boxtimes N^{\bullet}, R^{\bullet})\simeq
\Hom_{D(H_U^{\bullet})}(M^{\bullet}, \Hom^{\boxtimes}(N^{\bullet}, R^{\bullet})).
$$
\end{proposition}

\begin{proof}
First observe that the left-hand side can be rewritten as
\begin{align*}
\Hom_{D(H_U^{\bullet})}(M^{\bullet} \boxtimes N^{\bullet}, R^{\bullet}) &\simeq \Hom_{D(H_U^{\bullet})}(M^{\bullet} \boxtimes N^{\bullet}, (q_H \circ \mathbf{p})(R^{\bullet})) \\
& \simeq  \Hom_{K(H_U^{\bullet})}(\mathbf{p} (M^{\bullet} \boxtimes N^{\bullet}), \mathbf{p}R^{\bullet}) \\
& \simeq  \Hom_{K(H_U^{\bullet})}(\mathbf{p} \circ q_H\big[H_U^{\bullet}(2) \otimes_{H_U^{\bullet}\otimes_k H_U^{\bullet}} \mathbf{p}M^{\bullet}\otimes_k  \mathbf{p}N^{\bullet}\big]), \mathbf{p}R^{\bullet}).
\end{align*}
Next note that there is a canonical isomorphism of left $H_U^{\bullet}$-modules
$$
H_U^{\bullet}(2) \otimes_{H_U^{\bullet}\otimes_k H_U^{\bullet}} \mathbf{p}M^{\bullet}\otimes_k  \mathbf{p}N^{\bullet} \overset{\sim}{\longrightarrow}
\big(H_U^{\bullet}(2) \otimes_{H_U^{\bullet}} \mathbf{p}N^{\bullet}\big)\otimes_{H_U^{\bullet}} \mathbf{p}M^{\bullet}.
$$
(The initial tensor product in the target is via the second factor of $H_U^{\bullet}\otimes_k H_U^{\bullet}$. The last tensor product is via the first factor.) By \cite[Cor.~10.12.4.4]{BL}
$\mathbf{p}M^{\bullet}$ is homotopically flat, so
$$
\mathbf{p}\circ q_H \bigg[ H_U^{\bullet}(2)\otimes_{H_U^{\bullet}}\mathbf{p}N^{\bullet} \bigg] \otimes_{H_U^{\bullet}} \mathbf{p}M^{\bullet} \longrightarrow
\big(H_U^{\bullet}(2) \otimes_{H_U^{\bullet}} \mathbf{p}N^{\bullet}\big)\otimes_{H_U^{\bullet}} \mathbf{p}M^{\bullet}
$$
is a homotopically projective resolution of the right-hand side. We thus have a unique isomorphism in $K(H_U^{\bullet})$ between the two resolutions
$$
\mathbf{p}\circ q_H \bigg[ H_U^{\bullet}(2)\otimes_{H_U^{\bullet}}\mathbf{p}N^{\bullet} \bigg] \otimes_{H_U^{\bullet}} \mathbf{p}M^{\bullet} \overset{\sim}{\longrightarrow}
\mathbf{p} \circ q_H\big[H_U^{\bullet}(2) \otimes_{H_U^{\bullet}\otimes_k H_U^{\bullet}} \mathbf{p}M^{\bullet}\otimes_k  \mathbf{p}N^{\bullet}\big].
$$
This allows us to continue our calculation:
\begin{align*}
 \Hom_{D(H_U^{\bullet})}(M^{\bullet} \boxtimes N^{\bullet}, R^{\bullet}) & \simeq
 \Hom_{K(H_U^{\bullet})}(\mathbf{p}\circ q_H \bigg[ H_U^{\bullet}(2)\otimes_{H_U^{\bullet}}\mathbf{p}N^{\bullet} \bigg] \otimes_{H_U^{\bullet}} \mathbf{p}M^{\bullet} , \mathbf{p}R^{\bullet})\\
 &\simeq \Hom_{K(H_U^{\bullet})}(\mathbf{p}M^{\bullet} , \Hom_{H_U^{\bullet}}\big(\mathbf{p}\circ q_H \bigg[ H_U^{\bullet}(2)\otimes_{H_U^{\bullet}}\mathbf{p}N^{\bullet} \bigg], \mathbf{p}R^{\bullet}\big)) \\
 & \simeq \Hom_{D(H_U^{\bullet})}(M^{\bullet} , q_H \Hom_{H_U^{\bullet}}\big(\mathbf{p}\circ q_H \bigg[ H_U^{\bullet}(2)\otimes_{H_U^{\bullet}}\mathbf{p}N^{\bullet} \bigg], \mathbf{p}R^{\bullet}\big)) \\
 &=\Hom_{D(H_U^{\bullet})}(M^{\bullet}, \Hom^{\boxtimes}(N^{\bullet}, R^{\bullet}))
\end{align*}
by the usual tensor-hom adjunction for dg-modules.
\end{proof}

The Yoneda lemma immediately implies the following consequence of Theorem \ref{tensor} and Proposition \ref{adjoint}.

\begin{corollary}
Let $V^{\bullet}$ and $W^{\bullet}$ be objects of $D(G)$. Then there is an isomorphism
$$
\Hom^{\boxtimes}(H(V^{\bullet}), H(W^{\bullet})) \overset{\sim}{\longrightarrow} H(R\underline{\Hom}(V^{\bullet}, W^{\bullet})).
$$
\end{corollary}

In particular this gives an alternative description of the duality functor $\Delta$ for dg-modules:

\begin{corollary}
$\Delta(M^{\bullet}) \simeq \Hom^{\boxtimes}(M^{\bullet},\mathbf{U})$.
\end{corollary}

\bigskip

\noindent {\it{E-mail addresses}}: {\texttt{pschnei@wwu.de}, {\texttt{csorensen@ucsd.edu}}

\bigskip

\noindent {\sc{Peter Schneider, Math. Institut, Universit\"a{}t M\"{u{nster, M\"{u}nster, Germany.}}

\bigskip

\noindent {\sc{Claus Sorensen, Dept. of Mathematics, UC San Diego, La Jolla, USA.}}


\begin{thebibliography}{B-GAL}

\bibitem[BL]{BL}
Bernstein J., Lunts V.: \emph{Equivariant sheaves and functors}. Springer LNM 1578.

%
%
%

\bibitem[DM]{DM}
Deligne, P., and Milne, J.S.: \emph{Tannakian Categories}, in Hodge Cycles, Motives, and Shimura Varieties, LNM 900, 1982, pp. 101-228.

%
%


\bibitem[Har]{Har}
Hartshorne R.: \emph{Residues and Duality}. Springer Lect.\ Notes Math.\ vol.\ 20, 1966

\bibitem[KM]{KM}
Kriz I., May J. P.: \emph{Operads, algebras, modules and motives}. Ast\'e{}risque, tome 233 (1995).

%
%
%
%
%

\bibitem[Mac]{Mac}
Mac Lane S.: \emph{Homology}. Classics in Mathematics. Springer, Fourth Printing 1994.

\bibitem[Nee]{Nee}
Neeman A.: \emph{Triangulated categories}. Annals of Math. Studies, Number 148, Princeton Univ. Press (2001).

%
%

\bibitem[OS]{OS}
Ollivier R., Schneider P.: \emph{The modular pro-$p$ Iwahori-Hecke $\Ext$-algebra}.
Proc.\ Symp.\ Pure Math.\ 101, 255 - 308 (2019)

%

\bibitem[DGA]{DGA}
Schneider P.: \emph{Smooth representations and Hecke modules in characteristic $p$}. Pacific J. Math.\ 279, 447-464 (2015)

\bibitem[Schn]{Schn}
Schn\"u{}rer O.: \emph{Six operations on dg enhancements of derived categories of sheaves}. Selecta Math. New Ser. 24 (2018), no. 3, 1805-1911.

\bibitem[ScSc]{ScSc}
Schneider P., Scherotzke S.: \emph{Derived parabolic induction}. Bulletin of the London Mathematical Society 54, 264-274 (2022).

\bibitem[SS]{SS}
Schneider P., Sorensen C.: \emph{Duals and admissibility in natural characteristic}. Represent. Theory 27 (2023), 30-50.

%
%

\bibitem[Vig]{Vig}
Vigneras M.-F.: \emph{Repr\'esentations $\ell$-modulaires d'un groupe r\'eductifs $p$-adiques avec $\ell \neq p$}. Progress Math.\ vol.\ 131, Birkh\"auser 1996

\bibitem[Yek]{Yek}
Yekutieli A.: \emph{Derived Categories}. Cambridge Univ.\ Press 2020

\end{thebibliography}
\end{document}